\documentclass[12pt,epsfig,amsfonts]{amsart} 
\usepackage{amsmath,amsthm,amssymb,amscd,epsfig}
\usepackage{graphicx}
\newtheorem*{theorema}{Theorem A}
\newtheorem*{theoremb}{Theorem B}
\newtheorem*{theoremc}{Theorem C}
\newtheorem*{theoremd}{Theorem D}

\newtheorem{lemma}{Lemma}[section]
\newtheorem{sublemma}[lemma]{Sublemma}

\newtheorem{cor}[lemma]{Corollary}
\newtheorem{prop}[lemma]{Lemma}
\newtheorem{definition}[lemma]{Definition}

\begin{document}
\author{hiroki takahasi}

\address{Department of Mathematics,
Keio University, Yokohama,
223-8522, JAPAN} 
\email{hiroki@math.keio.ac.jp}

\title[Lyapunov spectrum for H\'enon-like maps] 
{Lyapunov spectrum for H\'enon-like maps \\
at the first bifurcation}

\begin{abstract}
For a strongly dissipative H\'enon-like map at the first bifurcation parameter
at which the uniform hyperbolicity is destroyed by the formation of tangencies inside the limit set,
we effect a multifractal analysis, i.e.,
  decompose the set of non wandering points on the unstable manifold 
  into level sets of an unstable Lyapunov exponent, and 
  give a partial description of the Lyapunov spectrum which encodes this decomposition.
  We derive a formula for
the Hausdorff dimension of the level sets in terms of the
entropy and unstable Lyapunov exponent of invariant probability measures, and show the continuity of the Lyapunov spectrum.
We also show that the set of points for which the unstable Lyapunov exponents do not exist carries a full Hausdorff dimension.
\end{abstract}

\maketitle

\section{introduction}

In the study of chaotic dynamical systems, one often encounters invariant sets with complicated geometric structures.
The multifractal analysis treats the so-called multifractal decomposition of these sets,
and the associated multifractal spectrum which encodes the decomposition.
The goal is to relate the spectrum to other characteristics
of the system, such as entropy and Lyapunov exponents of invariant measures, 
and to study the regularity of the spectrum, for instance, convexity, smoothness  and analyticity. 
With this study one tries to get more refined descriptions of the dynamics than purely stochastic considerations.

The cases of conformal or uniformly hyperbolic systems
are well understood \cite{BarSch00,Ols03,Pes97,PesWei97,Wei99}, and a complete picture is emerging.
For one-dimensional maps, several progresses have been made to relax these assumptions:
allowing parabolic fixed points \cite{GelRam09,JohJorObePol10,Nak00}; allowing critical points 
\cite{Chu10,ChuTak13,GelPrzRam10,IomTod11,PrzRiv11}. 
Nevertheless, little is known on higher dimensional systems.
Indeed, one can mention interesting recent developments \cite{BarIom11,UrbWol08} on two-dimensional 
{\it parabolic horseshoes.}
In these papers, however, 
the existence of global continuous invariant foliations are assumed, which
allows one to reduce a considerable part of the analysis 
to one-dimensional dynamics.
To our knowledge, there is no previous result on the multifractal analysis of two-dimensional maps
having tangencies of invariant manifolds.
This type of maps admit no global continuous invariant foliation,
and so new arguments and ideas are necessary to reduce to one-dimensional dynamics.

In this paper we are concerned with a family of planar diffeomorphisms
\begin{equation}\label{henon}
f_a\colon(x,y)\in\mathbb R^2\mapsto(1-ax^2,0)+b\cdot\Phi(a,b,x,y),\quad a\in\mathbb R, \ 0<b\ll1,\end{equation}
where $\Phi$ is bounded continuous in $(a,b,x,y)$ and 
$C^2$ in $(a,x,y)$. 
We assume\footnote{Condition \eqref{jacobian} is used exclusively in the proof of Lemma \ref{global}. See \cite{SenTak2}.} there exists a constant $C>0$ such that for all $a$ near $2$
and small $b$,
\begin{equation}\label{jacobian}\|D\log|\det Df_{a}|\|\leq C.\end{equation}
This family of diffeomorphisms has a fundamental importance in the creation of the theory of non-uniformly hyperbolic strange attractors 
\cite{BenCar91,MorVia93,WanYou01}. 
A relevant problem is to study the dynamics at a \emph{first bifurcation parameter} $a^*=a^*(b)\in\mathbb R$.
This parameter does not belong to the parameter sets of positive Lebesgue measure constructed in \cite{BenCar91,MorVia93,WanYou01}, and
satisfy the following
properties  \cite{BedSmi06,CLR08,DevNit79,Tak13}:

\begin{itemize}

\item $a^*\to2$ as $b\to0$;

\item the non wandering
set of $f_a$ is a uniformly hyperbolic horseshoe for $a>a^*$ ;

\item for $a=a^*$ there is a single orbit of homoclinic or heteroclinic tangency involving (one of)
the two fixed saddles. The tangency is quadratic, and the family $\{f_a\}_{a\in\mathbb R}$ unfolds this tangency generically.

\end{itemize}
Let $P$, $Q$ denote the fixed saddles of $f$ near $(1/2,0)$, $(-1,0)$ respectively.
The orbit of tangency intersects a small neighborhood of the origin exactly at one point, denoted by $\zeta_0$ (FIGURE 1).
If $f_{a^*}$ preserves orientation, then $\zeta_0\in W^s(Q)\cap W^u(Q)$.
If $f_{a^*}$ reverses orientation, then $\zeta_0\in W^s(Q)\cap W^u(P)$.
The map $f_{a^*}$ falls into the class of {\it non-uniformly hyperbolic systems}.
The sole obstruction to the uniform hyperbolicity is the orbit of the tangency $\zeta_0$.

The aim of this paper is to perform the multifractal analysis of $f_{a^*}$, in particular to
study its \emph{Lyapunov spectrum}.
  Although some aspects of the dynamics of $f_{a^*}$ resemble the horseshoe before the first bifurcation,
the presence of tangency is an intrinsic hurdle 
for understanding the global dynamics.

\begin{figure}
\begin{center}
\includegraphics[height=6.5cm,width=14cm]
{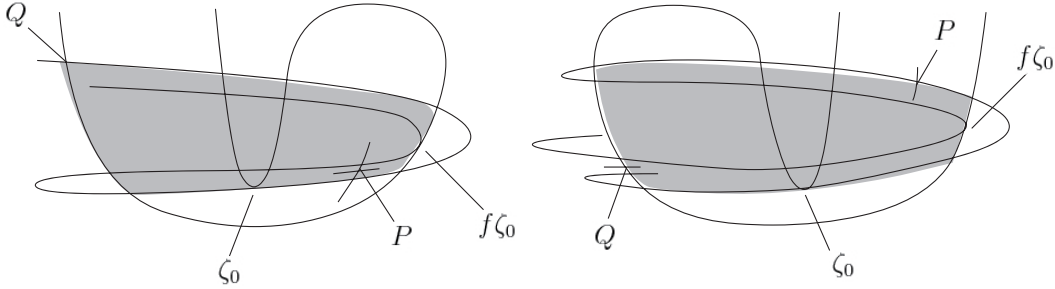}
\caption{Manifold organization for $a=a^*$:
orientation preserving/reversing cases (left/right).
The shaded domains represent the rectangle $R$ (see Sect.\ref{family})
containing the non wandering set $\Omega$.}
\end{center}
\end{figure}

We state our settings in more precise terms.
Write $f$ for $f_{a^*}$. 
At a point $x\in\mathbb R^2$ define a
one-dimensional subspace $E_x^u$ of $T_x\mathbb R^2$ 
which is exponentially contracted by backward iterates:
\begin{equation*}
\limsup_{n\to\infty}\frac{1}{n}\log\|D_xf^{-n}|E^u_x\|<0.\end{equation*}
Since $f^{-1}$ expands area, the one-dimensional subspace of $T_x\mathbb R^2$ with this property
is unique, when it makes sense. We call $E^u_x$ an \emph{unstable direction} at $x$, and
define an {\it unstable Jacobian} at $x$ by
$J^u(x)=\Vert D_xf|E^u_x\Vert.$
Let $\Omega$ denote
the non wandering set of $f$, which is a compact set. 
By a result of \cite{SenTak1}, $E^u_x$ makes sense for any $x\in\Omega$, and 
$x\mapsto E_x^u$ is continuous on $\Omega$ except at $Q$ where it is merely measurable.

For $x\in\Omega$ define
$$\underline{\lambda}^u(x)=\liminf_{n\to\infty}\frac{1}{n}\sum_{i=0}^{n-1}\log J^u(f^ix)
\ \text{ and }\ \bar{\lambda}^u(x)=\limsup_{n\to\infty}\frac{1}{n}\sum_{i=0}^{n-1}\log J^u(f^ix).$$
If both values coincide, then call this common value an
\emph{unstable Lyapunov exponent at $x$} and
denote it by $\lambda^u(x)$.
Since
the (non-uniform) expansion along the unstable direction is responsible for the chaotic behavior,
 the distribution of the unstable Lyapunov exponent 
 is important for understanding the dynamics of $f$.

If $f$ preserves orientation, let $W^u=W^u(Q)$. Otherwise,
let $W^u=W^u(P)$.
A good deal of information is contained in the unstable slice
$$\Omega^u=\Omega\cap W^u.$$
For each $\beta\in\mathbb R$ consider the level set
$$\Omega^u(\beta)=\left\{x\in \Omega^u\colon\text{$\lambda^u(x)$ is defined and }\lambda^u(x)=\beta\right\}.$$
The first question to ask is what are the values of $\beta$ for which $\Omega^u(\beta)\neq\emptyset$.
For uniformly hyperbolic systems as in the case $a>a^*$, such values are all positive and form a compact interval.
One can easily see that this is not the case for $f=f_{a^*}$, because $\lambda^u(\zeta_0)<0$.

Let 
$\mathcal M(f)$ denote the set of $f$-invariant
Borel probability measures.
An {\it unstable Lyapunov exponent} of a measure $\mu\in\mathcal M(f)$ is the number $\lambda^u(\mu)$
defined by 
$$\lambda^u(\mu)=\int\log J^ud\mu.$$
Set
$$\lambda_m^u=\inf\{\lambda^u(\mu)\colon \mu\in\mathcal M(f)\}\ \text{ and }\ \lambda_M^u=\sup\{\lambda^u(\mu)\colon \mu\in\mathcal M(f)\}.$$
By a result of \cite{CLR08}, $\lambda_m^u>0$. 
Since any measure is supported on the compact set $\Omega$,
$\lambda_M^u<\infty$.
 Set
$I=[\lambda_m^u,\lambda_M^u].$

\begin{theorema}
Let $b>0$ be sufficiently small and
 $f=f_{a^*(b)}$ as above. Then
 $\Omega^u(\beta)\neq\emptyset$ if and only if 
$\beta\in\{\lambda^u(\zeta_0)\}\cup I$.
 \end{theorema}
 
 The number
$\lambda^u(\zeta_0)$ equals the stable Lyapunov exponent
of the Dirac measure at $Q$, and so $\lambda^u(\zeta_0)\to-\infty$ as $b\to0$.
 The interval $I$ does not degenerate to a point as $b\to0$,
because the unstable Lyapunov exponents of the Dirac measures 
at $P$ and $Q$ converge to $\log2$ and
$\log4$ respectively. 
In fact, one can show that $\lambda_m^u\to\log2$ and $\lambda_M^u\to\log4$ as $b\to0$.

A proof of Theorem A relies on the fact that
$a^*\to2$ as $b\to0$, and so $f=f_{a^*}$ may be viewed as a singular perturbation of the endomorphism $(x,y)\mapsto (1-2x^2,0)$.
 However, the multifractal picture is quite in contrast to that of the quadratic map $x\in[-1,1]\to1-2x^2$.
The Lyapunov exponent of the quadratic map takes only three values: it is $\log 4$ at the repelling fixed point $-1$ and its preimage $1$,
$-\infty$ at the preimages of $0$, and is $\log 2$ at all other well-defined points.

Now, consider a multifractal decomposition
\begin{equation*}\label{decomposition}\Omega^u
=\left(\bigcup_{\beta\in\{\lambda^u(\zeta_0)\}\cup I}
\Omega^u(\beta)\right)\cup \hat \Omega^u,\end{equation*} 
where 
$\hat \Omega^u$ denotes the set of those $x\in\Omega^u$ for which 
$\underbar{$\lambda$}^u(x)\neq\bar{\lambda}^{u}(x)$ and so $\lambda^u(x)$ is undefined.
This decomposition has an extremely complicated topological structure. One can show that if $\beta\in  I$, then
$\Omega^u(\beta)$ is dense in $\Omega^u$ with respect to the induced topology on $W^u$.


To evaluate the size of each level set we adopt the Hausdorff dimension on $W^u$ defined as follows.
 Given $p\in(0,1]$
 the unstable Hausdorff 
$p$-measure of a set $A\subset W^u$ is defined by
$$m_p^u(A)=\lim_{\varepsilon\to0}\left(\inf\sum_{U\in\mathcal
U } {\rm length}(U)^p \right),$$ where ${\rm length}(\cdot)$ denotes the length on $W^u$ with respect to the induced Riemannian metric,
and the infimum is taken over all
countable coverings $\mathcal U$ of $A$ by open sets of $W^u$ with
length $\leq\varepsilon$. The unstable Hausdorff dimension of $A$,
denoted by $\dim_H^u$, is the unique number in $[0,1]$ such that
$$\dim_H^u(A)=\sup\{p\colon
m_p^u(A)=\infty\}=\inf\{p\colon m_p^u(A)=0\}.$$

\begin{figure}
\begin{center}
\includegraphics[height=6cm,width=9cm]
{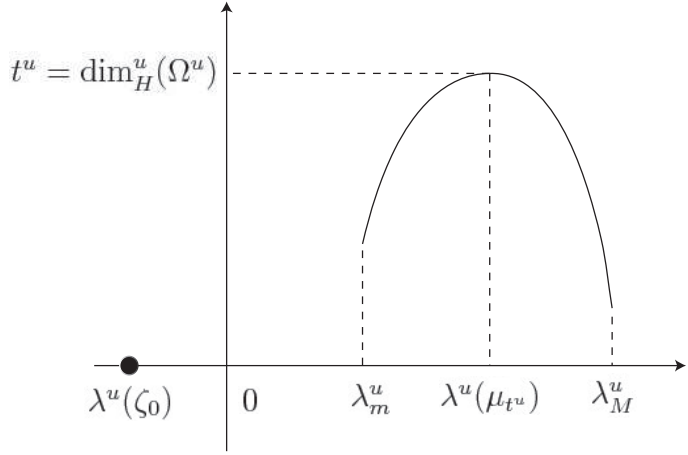}

\caption{Schematic picture of the graph of the Lyapunov spectrum $L^u\colon\{\lambda^u(\zeta_0)\}\cup I\to\mathbb R$.}

\end{center}
\end{figure}

\noindent Set
$$L^u(\beta)=\dim_H^u(\Omega^u(\beta)).$$
The object of our study is
the function $\beta\mapsto L^u(\beta)$, called a {\it Lyapunov spectrum}.

We give a formula for $L^u(\beta)$
in terms of the unstable Lyapunov exponents and entropy of invariant probability measures.
The entropy of $\mu\in\mathcal M(f)$ is denoted by  $h(\mu)$.
\begin{theoremb}
For any
$\beta\in I$, 
\begin{equation*}L^u(\beta)=\lim_{\varepsilon\to0}{\sup}\left\{\frac{h(\mu)}{\lambda^u
(\mu)}\colon\mu\in\mathcal M(f),\ \left|\lambda^u(\mu)
-\beta\right|<\varepsilon\right\}.\end{equation*}
 \end{theoremb}
Due to the existence of tangency, the unstable Lyapunov exponent as a function of measures may not be lower semi-continuous.
Hence, the limit in $\varepsilon$ is necessary.
 A formula similar to the one in Theorem B 
was obtained in \cite{ChuTak13} for a positive measure set of quadratic maps $x\in[-1,1]\to 1-ax^2$, but only for the time averages of continuous functions.




 We now move on to properties of the Lyapunov spectrum.
  Let us recall the thermodynamic formalism of $f$
 developed in \cite{SenTak1,SenTak2}.
 For $t\in\mathbb R$ define
 $$P(t)=\sup\left\{h(\mu)-t\lambda^u(\mu)\colon\mu\in\mathcal M(f)\right\}.$$
 A measure which attains this supremum is called an \emph{equilibrium measure} for $-t\log J^u$.
The function $t\mapsto P(t)$ is convex. One has $P(0)>0$, and Ruelle's inequality \cite{Rue78} gives $P(1)\leq0$.
Since $f$ has no SRB measure \cite{Tak12}, $P(1)<0$ holds. Hence the equation $P(t)=0$ has a unique solution in $(0,1)$,
denoted by $t^u$.
 There exists a unique equilibrium measure for $-t^u\log J^u$
 (\cite[Theorem A]{SenTak2}),
 denoted by $\mu_{t^u}$, and $t^u=\dim_H^u(\Omega^u)$, $t^u\to1$ as $b\to0$ (\cite[Theorem B]{SenTak2}).

\begin{theoremc}
The following holds for the function
$\beta\in I\mapsto  L^u(\beta)$:
\begin{itemize}
\item[(a)] it is continuous;
\item[(b)] increasing on $[\lambda_m^u,\lambda^u(\mu_{t^u})]$ and
 decreasing on $[\lambda^u(\mu_{t^u}),\lambda_M^u]$;
 \item[(c)] strictly positive in the interior of $I$;
\item[(d)]  $L^u(\beta)=t^u$ if and only if $\beta=\lambda^u(\mu_{t^u})$.
 \end{itemize}
 \end{theoremc}
 Theorem C illustrates what is sometimes called a {\it multifractal miracle}.
 Even though the multifractal decomposition is topologically complicated, the Lyapunov spectrum 
 which encodes the decomposition is continuous, and has several additional properties.
 
 \medskip
 
\noindent{\it Remark.}
From Theorem C(b), the minimum of $L^u$ is attained at the boundary of $I$. 
It is not known if the minimum is strictly positive.
Nor the convexity of the Lyapunov spectrum is known (See FIGURE 2 with care).
\medskip

The last theorem states that $\hat\Omega^u$ carries a full Hausdorff dimension.
For the subshift of finite type it is known  \cite{BarSch00} that
the set of irregular points for which the time averages of a given continuous function do not converge
carries the full dimension.
Since $\log J^u$ is not continuous, the same argument does not work in our setting.

 \begin{theoremd}
$\dim_H^u(\hat\Omega^u)=t^u.$
\end{theoremd}

To handle the two-dimensional dynamics of $f$ without uniform hyperbolicity, 
a basic idea is to use a (locally defined) stable foliation to identify points on the same leaf
(called {\it long stable leaves} in our terms, see Sect.\ref{stable}), and to recover the one-dimensional argument
 \cite{Chu10} as much as possible.
 Since the stable foliation is not globally defined,
 it is not possible to tell whether such a leaf through a given point exist.
To bypass this difficulty we proceed in three steps:
 
 \begin{itemize}
\item introduce critical points (Sect.\ref{critical}) in the spirit of Benedicks and Carleson \cite{BenCar91};
 
\item formulate a condition in terms of the speed of recurrence to the critical set, which is sufficient 
for the existence of the long stable leaf (Sect.\ref{controlled} and Sect.\ref{stable}):
 
\item show that the unstable Lyapunov exponent does not exist 
 at any point for which this condition fails (Sect.\ref{negligible}).
 \end{itemize}
 
 The rest of this paper consists of two sections.
In Sect.2 we collect mainly from \cite{SenTak1,SenTak2} and prove some results which will be
needed later.
In Sect.3 we bring them together and prove the theorems.


\section{Preliminaries}
In this section we collect from \cite{SenTak1,SenTak2} and prove some results which will be used in the proofs of the theorems.

\subsection{Constants}
Throughout this paper we shall be concerned with positive constants $\lambda$, $\delta$, $b$, the purposes of which are as follows:

\begin{itemize}

\item $\lambda$ is used to evaluate the rate of expansion of derivatives away 
from the point $\zeta_0$ of tangency (See Lemma \ref{hyperbolic});

\item $\delta$ determines the size of a neighborhood of $\zeta_0$ (See Sect.\ref{hipe});

\item  $b$ determines the magnitude of the reminder term $b\cdot\Phi$ in \eqref{henon}.
\end{itemize}

The $\lambda$ is a fixed constant in $(0,\log2)$. 
The $\delta$ and $b$ are small constants chosen in this order.
 The letter $C$ is used to denote any positive constant which is independent of 
 $\delta$ or $b$.

\subsection{The non wandering set}\label{family}

By a {\it rectangle} we mean any
compact domain bordered by two compact curves in $W^u$ and two in the
stable manifolds of $P$ or $Q$. By an {\it unstable side} of a
rectangle we mean any of the two boundary curves in $W^u$. A {\it
stable side} is defined similarly.

By the results of \cite{SenTak1} there exists a rectangle $R$ contained in the set 
$\{(x,y)\in \mathbb R^2\colon |x|<2, |y|< \sqrt{b}\}$ with the following properties (See FIGURE 1):

\begin{itemize}
\item $\displaystyle{\Omega=\{x\in R\colon f^nx\in R\ \text{ for every }n\in\mathbb Z\}}$;

\item one of the unstable sides of $R$ contains $\zeta_0$;

\item 
one of the stable sides of $R$ contains $f\zeta_0$. This side is denoted by $\alpha_0^+$. The other side, denoted by $\alpha_0^-$,  contains $Q$;

\item $f\alpha_0^+\subset\alpha_0^-$.
\end{itemize}


\subsection{Dynamics outside of critical region}\label{hipe}
Set $$I(\delta)=\{(x,y)\in R\colon |x|<\delta\}.$$
Observe that $\zeta_0\in I(\delta)$.
The next two lemmas state that the dynamics outside of $I(\delta)$ is ``uniformly 
hyperbolic" and no critical behavior occurs.
A slope $s(v)$ of a nonzero tangent vector $v=\left(\begin{smallmatrix}\xi\\\eta\end{smallmatrix}\right)$ at a point in $\mathbb R^2$ is defined by 
$s(v)=|\eta|/|\xi|$ if $\xi\neq0$, and
$s(v)=\infty$ if $\xi=0$.

\begin{lemma}\label{hyperbolic}
For any $\lambda\in(0,\log2)$ and $\delta\in(0,1)$ there exists $b>0$ such that the following holds for $f=f_{a^*(b)}$:
 If $n\geq1$ and $x\in R$ are such that $x,fx,\ldots,f^{n-1}x\notin I(\delta)$, then 
for any nonzero tangent vector $v$ at $x$ with $s(v)\leq\sqrt{b}$,
\begin{itemize}

\item[(a)]  $\|D_xf^nv\|\geq \delta e^{\lambda n}.$
 If, in addition $f^nx\in I(\delta)$, then $\|D_xf^nv\|\geq e^{\lambda n}$;

\item[(b)] 
$s(D_xf^nv)\leq\sqrt{b}$.

\end{itemize}
\end{lemma}
\begin{proof}
From  the fact that $f$ may be viewed as a small perturbation of
the map $x\mapsto 1-2x^2$. \end{proof}

\begin{lemma}\label{curvature}{\rm (\cite[Lemma 2.3]{Tak11})}
Let $\gamma$ be a $C^2$ curve in $R$ and $x\in\gamma$. For each $i\geq0$ let $\kappa_i(x)$ denote the curvature of $f^i\gamma$ at $f^ix$.
Then
$$\kappa_i(x)\leq\frac{(Cb)^i}{\|D_xf^i|T_x\gamma\|^3}\kappa_0(x)+\sum_{j=1}^i\frac{(Cb)^j}{\|D_{f^{i-j}x}f^j|T_{f^{i-j}x}f^{i-j}\gamma\|^3}.$$
\end{lemma}
By a \emph{$C^2(b)$-curve} we mean a compact, nearly horizontal $C^2$ curve in $R$ such that the slopes of its tangent directions are $\leq\sqrt{b}$ and the
curvature is everywhere $\leq\sqrt{b}$. 
\begin{lemma}\label{sae} If $\gamma$ is a $C^2(b)$-curve in $R$ not intersecting $I(\delta)$, then $f\gamma$ is a $C^2(b)$-curve. 
\end{lemma}

\begin{proof}
From Lemma \ref{hyperbolic} and Lemma \ref{curvature}.
\end{proof}

\subsection{Critical points}\label{critical}
Returns to the inside of $I(\delta)$ are inevitable and must be treated with care.
A key ingredient is the notion of critical points, i.e.,
points of tangencies between $C^2(b)$-curves in $W^u$ and preimages of leaves
of a stable foliation. 
We quote results from \cite{SenTak1} surrounding critical points, and develop them 
slightly further.

From
the hyperbolicity of the saddle $Q$,
there exist two mutually disjoint connected open sets $U^-$, $U^+$ independent of $b$ such that
$\alpha_0^-\subset U^-$, $\alpha_0^+\subset U^+$, $U^+\cap fU^+=\emptyset=U^+\cap fU^-$ and 
a foliation $\mathcal F^s$ of $U=U^-\cup U^+$ by one-dimensional leaves such
that: 
\begin{itemize}
\item $\mathcal F^s(Q)$, the leaf of $\mathcal F^s$ containing $Q$,
contains $\alpha_0^-$; 
\item if $x,fx\in U$, then $f(\mathcal F^s(x))
\subset\mathcal F^s(fx)$;

\item Let $e^s(x)$ denote the unit vector in $T_x\mathcal F^s(x)$ whose second component is positive. 
Then $x\mapsto e^s(x)$ is $C^{1}$, $\|D_xfe^s(x)\|\leq Cb$ and $\|D_xe^s(x)\|\leq C$;

\item If $x,fx\in U$, then $s(e^s(x))\geq
C/\sqrt{b}.$
\end{itemize}
\begin{definition}
{\rm We say $\zeta\in W^u\cap I(\delta)$ is a {\it critical point} if $f\zeta\in U^+$ and
$T_{f\zeta}W^u=T_{f\zeta}\mathcal F^s(f\zeta)$.}
\end{definition}

From the first two conditions on $\mathcal F^s$ and $f\alpha_0^+\subset\alpha_0^-$, there is a leaf of $\mathcal F^s$ which 
contains $\alpha_0^+$. Since $f\zeta_0\in\alpha_0^+$ we have 
$f\zeta_0\in U^+$ and
$T_{f\zeta_0}W^u=T_{f\zeta_0}\mathcal F^s(f\zeta_0)$,
namely, $\zeta_0$ is a critical point.
The next lemma tells about the location of
all other critical points.
Let $S$ denote
the compact lenticular domain bounded by the parabola
$f^{-1}\alpha_0^+\cap R$ and the unstable side of $R$ not containing $\zeta_0$.
\begin{lemma}\label{cexist}
Let $\gamma$ be a $C^2(b)$-curve in $I(\delta)$ stretching across $I(\delta)$.
Then there exists a unique critical point $\zeta\in\gamma$. In addition, $\zeta\in S$.
if $\zeta\neq\zeta_0$ then $\zeta\in {\rm int}S$.
\end{lemma}
\begin{proof}
We claim that any leaf of $\mathcal F^s$
at the right of the one containing $\alpha_0^+$
is tangent to $f\gamma$ and the tangency is quadratic, or else
it intersects $f\gamma$ exactly at two points.
This follows from \cite[Lemma 2.2]{Tak11}, the 
uniform boundedness of $\|D_xe^s(x)\|$ and $s(e^s(x))$.
Hence there exists a critical point on $\gamma$.
If $\zeta_1$, $\zeta_2$ are distinct critical points on $\gamma$, then
the leaves $\mathcal F^s(f\zeta_1)$, $\mathcal F^s(f\zeta_1)$ must intersect each
other, which is a contradiction. Hence the uniqueness holds.
Since the quadratic tangency occurs on or at the right of $\alpha_0^+$,
the last two statements hold.
\end{proof}

By Lemma \ref{cexist}, any critical point other than $\zeta_0$ is contained in the interior of $S$, so that
it is mapped to the outside of $R$, and then escape to infinity
under forward iteration.
Hence, the critical orbits are contained in a region where 
the uniform hyperbolicity is apparent. By binding generic orbits which fall inside $I(\delta)$
to suitable critical points, and then copying the exponential growth along the critical orbits,
one shows that the horizontal slopes and the expansion are restored
after suffering from the loss due to the folding behavior near $I(\delta)$.

In the next lemma we assume $\delta>0$ is sufficiently small.
Let $\zeta$ be a critical point and $x\in I(\delta)\setminus S$.
We say a unit tangent vector $v$ at $x$
is {\it in admissible position relative to} $\zeta$ if there exists a $C^2(b)$-curve 
which is tangent to both $T_\zeta W^u$ and $v$.
Set
\begin{equation}\label{cb}c(b)=-\frac{1}{\log b}.\end{equation}
Let us agree that for two positive real numbers $A$, $B$, $A\approx B$ indicates that both 
$A/B$, $B/A$ are bounded from above by a constant independent of $\delta$ or $b$.

\begin{lemma}\label{binding}
Let $\zeta$ a critical point, $x\in (\Omega\cap I(\delta))\setminus S$  and $v$ be a unit tangent vector at $x$
  in admissible position relative to $\zeta$.
there exist positive integers $p=p(\zeta,x),q=(\zeta,x)$ such that:
\begin{itemize}

\item[(a)] $q\leq -c(b)\log|\zeta-x|\ll -(2/3)\log |\zeta-x|\leq p$; 

\item[(b)] $f^i\zeta$, $f^ix\in U$ for every $1\leq i\leq p$;

\item[(c)] $s(D_xf^pv)\leq\sqrt{b}$ and
$\|D_xf^pv\|\geq e^{\frac{\lambda}{3}p}$;

\item[(d)] $\|D_xf^qv\|\leq C|\zeta-x|^{1-c(b)};$

\item[(e)] $\|D_xf^iv\|<1$ for every $1\leq i<q$ and
$\|D_xf^iv\|\approx 2|\zeta-x|\cdot\|D_{fx}f^{i-1}\left(\begin{smallmatrix}1\\0\end{smallmatrix}\right)\|$ for every $q\leq i\leq p$.

\end{itemize}
\end{lemma}

\begin{proof}
We only give a proof of (d). The rest of the items is contained in \cite[Lemma 2.5]{SenTak1}.
Split $D_xfv=A\cdot\left(\begin{smallmatrix}1\\0\end{smallmatrix}\right)+B\cdot e^s(fx),$
 $A$, $B\in\mathbb R$.
Since the forward orbit of $f\zeta$ does not intersect $I(\delta)$, the tangent vector 
$\left(\begin{smallmatrix}1\\0\end{smallmatrix}\right)$ at $f\zeta$
grows exponentially in norm under forward iteration.
Since the forward orbit of $fx$ shadows that of $f\zeta$, 
$\|D_{fx}f^{q-1}\left(\begin{smallmatrix}1\\0\end{smallmatrix}\right)\|\approx\|D_{f\zeta}f^{q-1}\left(\begin{smallmatrix}1\\0\end{smallmatrix}\right)\|$
holds.
 From the quadratic behavior near the critical point we have $|A|\approx |\zeta-x|$. 
Then, $q\ll p$ in Lemma \ref{binding}(a) and the exponential contraction of $e^s(fx)$ implies
$|A|\cdot \|D_{fx}f^{q-1}\left(\begin{smallmatrix}1\\0\end{smallmatrix}\right)\|\gg |B|\cdot \|D_{fx}f^{q-1}e^s(fx)\|$.
Hence
$\|D_xf^{q}v\|\approx|\zeta-x|\cdot     \|D_{f\zeta} f^{q-1}\left(\begin{smallmatrix}1\\0\end{smallmatrix}\right)\|      \leq C|\zeta-x|^{1-c(b)},$
where the last inequality follows from 
the definition of $q$ in \cite[Sect.2.3]{SenTak1}. \end{proof}

\subsection{Existence of binding points}\label{stablegeo}
We look for suitable critical points for returns to $I(\delta)$ with the help of the nice geometry of $W^u$
which is particular to the first bifurcation parameter $a^*$.
Let $\alpha_1^+$ denote the connected component of $W^s(P)\cap R$ containing $P$,
and $\alpha_1^-$ the connected component of $f^{-1}\alpha_1^+\cap R$ not containing $P$. 
Let $\Theta$ denote the rectangle bordered by $\alpha_1^-$, $\alpha_1^+$ and the unstable sides of $R$.

\begin{lemma}\label{curvature2}
Let $\gamma$ be a $C^2(b)$-curve in $I(\delta)$ and suppose there exists a critical point on $\gamma$. 
If $n\geq1$ is such that $\Theta\cap f^i\gamma=\emptyset$ for $i=0,1,\ldots,n-1$ and $f^n\gamma\cap\Theta\neq\emptyset$, 
then any connected component of $\Theta\cap f^n\gamma$ is a $C^2(b)$-curve.
\end{lemma}
\begin{proof}
By Lemma \ref{sae} it suffices to show that for any $x\in\gamma$, 
$\|D_{f^ix}f^{n-i}|T_{f^ix}f^i\gamma\|\geq\delta$ for every $0\leq i\leq n-1$. This follows from 
Lemma \ref{hyperbolic} and Lemma \ref{binding}(e).
\end{proof}

Let $\tilde\Gamma^u$ denote the collection of connected 
components of $\Theta\cap W^u$ with respect to the intrinsic topology on $W^u$.
\begin{lemma}
\label{nogu} 
Any element of $\tilde\Gamma^u$ 
is a $C^2(b)$-curve with endpoints in $\alpha_1^-$, $\alpha_1^+$.
\end{lemma}
\begin{proof}
Let $\gamma$ denote the unstable side of $\Theta$ not containing $\zeta_0$.
This is $C^2(b)$, and contains a fundamental domain in $W^u$.
It suffices to show that for each $n\geq0$, 
any connected component of $\Theta\cap\bigcup_{i=0}^{n}f^i\gamma$ 
is a $C^2(b)$-curve with endpoints in $\alpha_1^-$, $\alpha_1^+$.
This holds for $n=0$. If it holds for $n=k$, then by Lemma \ref{curvature2},
any connected component of $\Theta\cap\bigcup_{i=0}^{k+1}f^i\gamma$ 
is $C^2(b)$. Since the endpoints of $\gamma$ are mapped to the stable sides of $\Theta$,
the statement holds for $n=k+1$.
\end{proof}

Define
$$\Gamma^u=\{\gamma^u\colon\text{$\gamma^u$ is
the pointwise limit of the sequence in $\tilde\Gamma^u$}\}.$$
Since elements of $\tilde\Gamma^u$ are $C^2(b)$ by Lemma \ref{nogu}, the pointwise
convergence is equivalent to the uniform convergence. Since curves in $\tilde\Gamma^u$ are pairwise disjoint,
the uniform convergence is equivalent to the $C^1$ convergence.
Hence, curves in $\Gamma^u$ are $C^1$ and the slopes of their
tangent directions are $\leq\sqrt{b}$.
Elements of $\Gamma^u$ are called {\it long unstable leaves}.
Set $$\mathcal W^u=\bigcup_{\gamma^u\in\Gamma^u}\gamma^u.$$

\begin{figure}
\begin{center}
\includegraphics[height=4cm,width=6cm]
{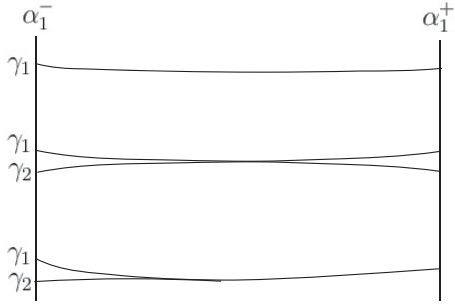}
\caption{The long unstable leaves.}
\end{center}
\end{figure}

Several remarks are in order on the long unstable leaves:

\begin{itemize}
\item each leaf is the
(strictly) monotone  limit of curves in $\tilde\Gamma^u$, so that
any connected component of $\mathcal W^u$ contains at most two leaves;

\item two intersecting leaves are tangent at every point of the intersection;

\item For $x\in\mathcal W^u$, $E_x^u=T_x\gamma^u$,
where $\gamma^u$ denotes any leaf containing $x$ (\cite[Lemma 3.2(P2)]{SenTak2});

\item $\Omega\cap \Theta\subset\mathcal W^u$ {\rm (\cite[Lemma 2.8]{SenTak1})}.

\end{itemize}

\begin{lemma}\label{eligible}
If $x\in \Omega\cap I(\delta)$, then there exists a critical point 
relative to which any unit vector spanning $E^u_x$ is in admissible position. 
\end{lemma}
\begin{proof}
A long stable leaf containing $x$ is 
accumulated in $C^1$ by curves in $\tilde\Gamma^u$, each of which
contains a critical point by Lemma \ref{cexist}.
\end{proof}

If $x\in \Omega\cap I(\delta)$, then critical points as in Lemma \ref{eligible} are not unique.
Let $\zeta(x)$ denote the one which is closest to the saddle in $W^u$ with respect to 
the induced metric on $W^u$, and call it a \emph{binding point} for $x$.
Write $p(x)=p(\zeta(x),x)$, $q(x)=q(\zeta(x),x)$
and call them the \emph{fold} and \emph{bound} periods of $x$.

\subsection{Bound-free structure}\label{bf}
To the forward orbit of $x\in\Omega$ we associate a sequence
$$0\leq n_1<n_1+p_1< n_2<n_2+p_2<n_3<\cdots$$ of integers which record the pattern of recurrence to $I(\delta)$ in the following manner.
First,
$n_1=\min\{n\geq0\colon f^nx\in I(\delta)\}$ and $p_1=p(f^{n_1}x)$. 
Given $n_k$ and $p_k$, set
$n_{k+1}=\min\{n\geq n_k+p_k\colon f^nx\in I(\delta)\}$ and $p_{k+1}=p(f^{n_{k+1}}x)$. 
This decomposes the forward orbit of $x$ into segments corresponding to time intervals
$(n_k,n_k+p_k)$ and $[n_k+p_k,n_{k+1}]$, during which we refer to the points in the orbit of $x$
as being ``bound" and ``free" respectively.
The $\{n_k\}_k$ are  the only return times to $I(\delta)$.


\subsection{Controlled points}\label{controlled}
For $x\in\Omega$ define 
$$d_{\rm crit}(x)=\begin{cases}|\zeta(x)-x|&\text{ if $x\in I(\delta)$};\\
1&\text{ otherwise,}\end{cases}$$
where $\zeta(x)$ is the binding point for $x$ determined in Sect.\ref{stablegeo}.

\begin{definition}
We say $x\in\Omega$ is {\it controlled} if $d_{\rm crit}(f^{n}x)> b^{\frac{n}{9}}$ holds for every $n\geq0$.
\end{definition}
The next lemma states that points without too deep returns to the criticality is controlled eventually.

\begin{lemma}\label{sr}
Let $m\geq0$.
If $d_{\rm crit}(f^{n}x)>  b^{\frac{n}{9}}$ for every $n\geq m$, then there exists $k\in[0,m]$ such that $f^kx$ is controlled.
\end{lemma}
\begin{proof}
The statement for $m=0$ is immediate from the definition.
Let $m=1$ and suppose that $f^kx$ is not controlled for every $k\in[0,m]$. 
Then, it is possible to define a sequence $\{k_i\}_{i=1}^s$ of nonnegative integers inductively as follows: 
$k_1=\min\{n\geq0\colon d_{\rm crit}(f^nx)\leq b^{\frac{n}{9}}\}.$
Since $x\in G_m$ we have $k_1<m$.
Given $k_1,\ldots,k_{i}$ with $k_1+\cdots+k_i<m$
and  $d_{\rm crit}(f^{k_1+\cdots+k_i}x)\leq  b^{\frac{k_i}{9}}$,
define $k_{i+1}=\min\{n>0\colon d_{\rm crit}(f^{k_1+\cdots+k_i+n}x)\leq b^{\frac{n}{9}}\}.$
We have $k_1+\cdots+k_{s-1}< m\leq k_1+\cdots+k_{s}$.
Since $b^{\frac{k_i}{9}}\cdot \|Df^{2k_i}\|\ll1$, $f^{k_1+\cdots+k_i}x$ shadows the forward orbit of the binding point 
at least up to time $2k_i$, and so
$2k_i<k_{i+1}$. This yields $k_1+\cdots+k_s<2k_s$, and thus
$d_{\rm crit}(f^{k_1+\cdots+k_s}x)\leq b^{\frac{k_s}{9}}<b^{\frac{2(k_1+\cdots+k_s)}{9}}.$
From the assumption on $x$ and $m\leq k_1+\cdots+k_{s}$ we have
$d_{\rm crit}(f^{k_1+\cdots+k_s}x) > b^{\frac{k_1+\cdots+k_s}{9}}.$
These two inequalities yield a contradiction.
\end{proof}

\subsection{Long stable leaves}\label{stable}
By a {\it vertical $C^2(b)$-curve} we mean a compact, nearly vertical $C^2$ curve in $R$
with endpoints
in the unstable sides of $R$, and of the form
$$\{(x(y),y)\colon |x'(y)|\leq C\sqrt{b}, |x''(y)|\leq
C\sqrt{b}\}.$$
A vertical $C^2(b)$-curve $\gamma^s$ is called a {\it long stable leaf} if 
for any $x,y\in\gamma^s$, $|f^nx-f^ny|\leq Cb^{\frac{n}{2}}$
holds for every $n\geq0$.

\begin{lemma}\label{leaf1}
If $x\in\Omega$ is controlled, then there exists a unique long stable leaf through $x$,
denoted by $\gamma^s(x)$. In addition, the following holds: 
\begin{itemize}
\item[(a)] 
for all $y$, $z\in\gamma^s(x)\cap\Omega$ and $n>0$,
$$\frac{ \|D_yf^n|E_y^u\|}{ \|D_zf^n|E_z^u\|}  \leq2;$$

\item[(b)] if $x,y\in\Omega$ are controlled, then the Hausdorff distance between $\gamma^s(x)$
and $\gamma^s(y)$ is $\leq e^{C\sqrt{b}}|x-y|$.
\end{itemize}
\end{lemma}
\begin{proof} 
In view of the results in \cite[Sect.6, Sect.7C]{MorVia93}, \cite[Lemma 2.4]{BenVia01}
 \cite[Sublemma A.2]{SenTak2}, it suffices to show the following expansion estimate:
\begin{equation}\label{appendix2}\|D_xf^n|E^u_x\|\geq b^{\frac{n}{10}}\text{ for every $n\geq1$.} 
\end{equation}

To show \eqref{appendix2} we introduce the bound/free 
structure on the orbit of $x$. If $f^{n}x$ is free, then 
the orbit $x,\ldots,f^nx$ is decomposed into alternate bound and free segments. Applying the expansion estimates in 
Lemma \ref{hyperbolic} and Lemma \ref{binding}
we have $\Vert D_xf^n|E^u_x\Vert\geq \delta e^{\frac{\lambda}{3}n}>
b^{\frac{n}{10}}$. If $f^{n}x$ is bound, then
 there exists an integer $0<m<n$ such that
$f^mx\in I(\delta)$ and $m<n<m+p$, where $p$ is the bound period of $f^mx$.
Since $f^{m+p}x$ is free and $\|Df\|<5$ we have
$\Vert D_xf^n|E^u_x\Vert\geq 5^{-(m+p-n)} \Vert
D_xf^{m+p}|E^u_x\Vert> 5^{-p}$. Since $x$ is controlled, $p\leq -(2n/27)\log b$ 
and so $\|D_xf^n|E^u_x\|\geq b^{\frac{2\log 5}{27}n}.$
\end{proof}








\subsection{Points with too deep returns are negligible}\label{negligible}

For each $m\geq0$ define 
$$G_m=
\{x\in\Omega\colon \text{$d_{\rm crit}(f^{n}x)> b^{\frac{n}{10}}$ for every $n\geq m$}\}.$$
Set
$$\Omega_*=\Omega\setminus\bigcup_{m=0}^\infty G_m.$$
This is the set of points which return to the deep inside of the criticality.
It is true that we lose control of derivatives on $\Omega_*$. However,
the next lemma states that unstable Lyapunov exponents are undefined on $\Omega_*$.
Hence, we may neglect $\Omega_*$ for our purpose.

\begin{lemma}\label{new}
If $x\in \Omega_*$, then $\underline{\lambda}^u(x)\neq \bar{\lambda}^u(x)$.
\end{lemma}
\begin{proof}
Consider the bound/free structure in Sect.\ref{bf} for the forward orbit of $x$. By definition, 
$d_{\rm crit}(f^nx)\leq b^{\frac{n}{10}}$ holds for infinitely many $n>0$.
For these $n$, $f^nx$ is free. By Lemma \ref{binding} and \eqref{cb}, the corresponding fold period $q=q(f^nx)$
satisfies
$$q\leq -c(b) d_{\rm crit}(f^{n}x)\leq-c(b)\frac{n}{10}\log b=\frac{n}{10}.$$
Hence $n+q\leq(11/10)n$, and by Lemma \ref{binding}(c),
$$\|D_{f^{n}x}f^{q}|E^u_{f^{n}x}\|\leq Cd_{\rm crit}(f^{n}x)^{1-c(b)}
\leq Cb^{\frac{(1-c(b))}{10}n}\leq Cb^{\frac{(1-c(b))10}{11}(n+q)}.$$
Hence we have 
 $$\|D_{x}f^{n+q}|E^u_{x}\|=\|D_{x}f^{n}|E^u_{x}\|\cdot\|D_{f^{n}x}f^{q}|E^u_{f^{n}x}\|<5^{n}\cdot 
 Cb^{\frac{(1-c(b))10}{11}(n+q)} <b^{\frac{n+q}{2}}.$$
Since this holds for infinitely many $n>0$, we obtain $\underline{\lambda}^u(x)\leq (1/2)\log b<0$.
On the other hand, decomposing the forward orbit of $x$ into alternate bound and free segments,
and then applying the expansion estimates in Lemma \ref{hyperbolic} and Lemma \ref{binding} imply 
$\bar{\lambda}^u(x)\geq\lambda/3>0.$
\end{proof}

\begin{cor}\label{G0}
For any $\mu\in\mathcal M(f)$, $\mu(\Omega_*)=0$.
\end{cor}
\begin{proof}
From the ergodic decomposition, it suffices to consider the case
where $\mu$ is ergodic. From the Ergodic Theorem, $\underline{\lambda}^u(x)= \bar{\lambda}^u(x)$
holds for $\mu$-a.e. $x$. Hence $\mu(\Omega_*)=0$.
\end{proof}

\begin{figure}
\begin{center}
\includegraphics[height=4cm,width=10cm]{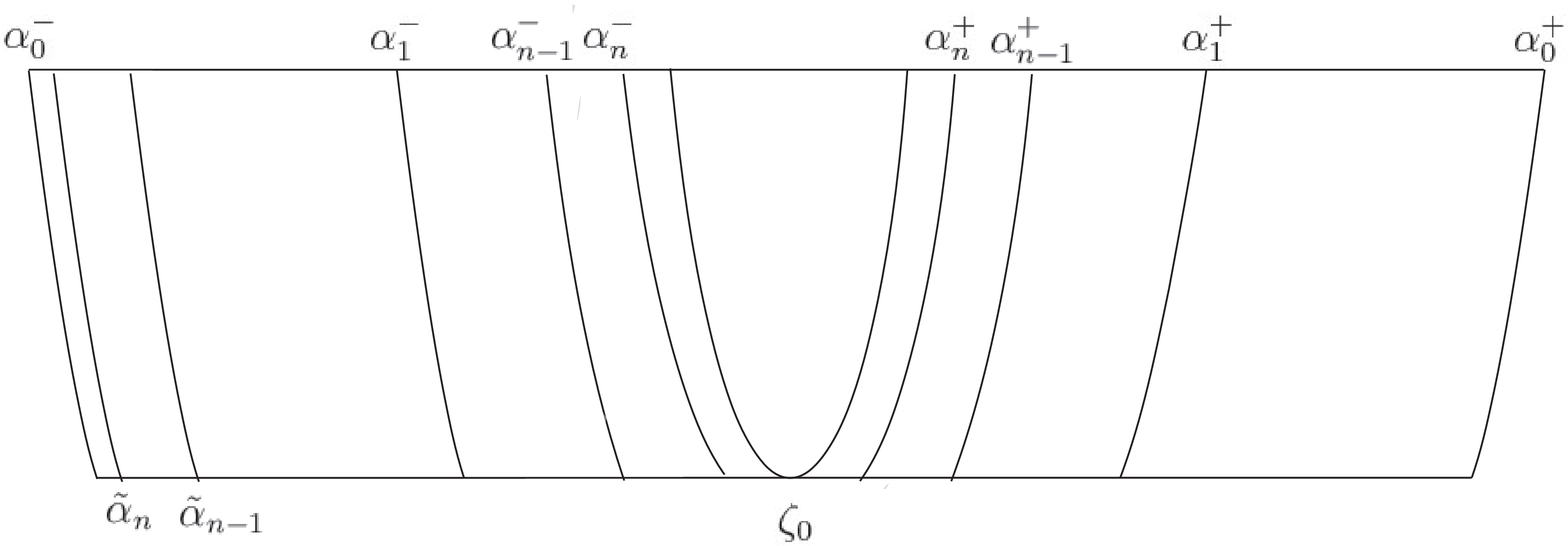}
\caption{The rectangle $R$ and the curves $\{\tilde\alpha_n\}$, $\{\alpha_n^{+}\}$, $\{\alpha_n^{-}\}$.
The $\{\tilde\alpha_n\}$ accumulate on the left stable side of $R$. Both $\{\alpha_n^{+}\}$ and $\{\alpha_n^{-}\}$ accumulate on the parabola $f^{-1}\alpha_0^+\cap R$
containing the point of tangency $\zeta_0$ near the origin.}
\end{center}
\end{figure}

\subsection{Inducing}\label{induced map}
We now recall the inducing construction performed in \cite{SenTak2}.
Define a sequence $\{\tilde\alpha_n\}_{n=0}^\infty$ of compact curves in $W^s(P)\cap R$ inductively as follows.
First, set $\tilde\alpha_0=\alpha_1^+$. Given $\tilde\alpha_{n-1}$, define $\tilde\alpha_n$ to be one of the two connected components of 
$R\cap f^{-1}\tilde\alpha_{n-1}$ which is at the left of $\zeta_0$. Observe that $\tilde\alpha_1=\alpha_1^-$. By the Inclination Lemma, 
the Hausdorff distance between $\tilde\alpha_n$ and $\alpha_0^-$ converges to $0$ as $n\to\infty$.

For each $n\geq0$ let $\alpha_n$ denote the connected component of $R\cap f^{-1}\tilde\alpha_n$ 
which is not $\tilde\alpha_{n+1}$. 
The set $R\cap f^{-1}\alpha_n$ consists of two curves, one at the left
of $\zeta_0$ and the other at the right. They are denoted by $\alpha_{n+1}^-$,
$\alpha_{n+1}^+$ respectively.
By definition, these curves obey the following diagram
\begin{equation*}
\{\alpha_{n+1}^-,\alpha_{n+1}^+\}\stackrel{f^2}{\to}\tilde\alpha_n
\stackrel{f}{\to}\tilde\alpha_{n-1}
\stackrel{f}{\to}\tilde\alpha_{n-2}\stackrel{f}{\to}\cdots
\stackrel{f}{\to}\tilde\alpha_1=\alpha_1^-\stackrel{f}{\to}
\tilde\alpha_0=\alpha_1^+.\end{equation*}

Define $r\colon \Theta\to \mathbb N\cup\{\infty\}$ by
$$r(x)=\inf(\{n>0\colon f^nx\in\Theta\}\cup\{\infty\}),$$
which is the first return time of $x$ to $\Theta$.
Note that:
\begin{itemize}

\item $r(x)=1$ if and only if $x\in\alpha_1^-\cup\alpha_1^+$;
$r(x)=n+1$ $(n\geq1)$ if and only if $x$ is sandwiched by $\alpha_{n}^+$ and $\alpha_{n+1}^+$, or
by $\alpha_{n}^-$ and $\alpha_{n+1}^-$;
 $r(x)=\infty$ if and only if $x\in S$;

\item each level set of $r$ except $S$ has exactly two connected components.

\end{itemize}

Let $\mathcal P$ denote the partition of the set $\Theta\setminus(S\cup\alpha_1^-\cup\alpha_1^+)$ into connected components of the level sets of 
the function $r$. 
The $\mathcal P$ is well-defined because 
$\alpha_n$ and $\alpha_0^+$ are long stable leaves, and the Hausdorff distance between 
them converges to $0$ as
$n\to\infty$ by Lemma \ref{leaf1}(b). 
 Set $\mathcal P_1=\{\omega=\overline{\eta}\colon\eta\in\mathcal P\}$,
where the bar denotes the closure operation.
For each $n\geq2$ define
$$\mathcal P_n=\left\{\omega_0\cap \bigcap_{i=1}^{n-1} f^{-r(\omega_0)}\circ f^{-r(\omega_1)}
\circ\cdots\circ f^{-r(\omega_{i-1})}\omega_i
\colon\omega_0,\omega_1,\ldots,\omega_{n-1}\in \mathcal P_1\right\}.$$
Elements of $\bigcup_{n\geq0}\mathcal P_n$ are called \emph{proper rectangles.}
It is easy to see the following holds:

\begin{itemize}

\item the unstable sides of a proper rectangle are formed by two curves contained in the unstable sides of $\Theta$.
Its stable sides are formed by two curves contained in $W^s(P)$;

\item two proper rectangles are either nested, disjoint, or intersect each other only at their common stable sides. 
\end{itemize}
On the interior of each $\omega\in\mathcal P_1$, the value of $r$ is constant. This value is denoted by $r(\omega)$.
For each $\omega\in\mathcal P_n$ 
define its \emph{inducing time} $\tau(\omega)$ by
\begin{equation}\label{tau}\tau(\omega)=\sum_{i=0}^{n-1}r(\omega_i).\end{equation}
It is easy to see the following holds:
\begin{itemize}

\item
the unstable sides of $f^{\tau(\omega)}\omega$ are formed by two curves in $\tilde\Gamma^u$.
Its stable sides are formed by two curves contained in the stable sides of $\Theta$ (See FIGURE 5);

\item let $k\in(0,\tau(\omega))$. Then
${\rm int}\Theta\cap f^k\omega\neq\emptyset$ if and only if $k=r(\omega_0)+\cdots+r(\omega_i)$ for some $i\in[0,n-1]$.

\end{itemize}

\begin{lemma}\label{global}
For any $\gamma^u\in\Gamma^u$ and any proper rectangle $\omega$,
$\gamma^u\cap\omega$ is a compact curve joining the stable sides of $\omega$. In addition,
\begin{itemize}

\item[(a)] 
$\displaystyle{\sup_{x\in\gamma^u\cap\omega}\|D_xf^{\tau(\omega)}|E_x^u\|\geq e^{\frac{\lambda}{3}\tau(\omega)}};$

\item[(b)] $\displaystyle{\sup_{x,y\in\gamma^u\cap\omega}
\frac{\|D_yf^{\tau(\omega)}|E_y^u\|}{\|D_xf^{\tau(\omega)}|E_x^u\|}        \leq
C|f^{\tau(\omega)}x-f^{\tau(\omega)}y|}.$

\end{itemize}
\end{lemma}
\begin{proof}
From the first property of the proper rectangles and Lemma \ref{nogu},
any curve in $\tilde\Gamma^u$ intersects any of the stable sides of $\omega$ exactly at one point,
and this intersection is transverse.
Since $\gamma^u$ is a $C^1$-limit of curves in $\tilde\Gamma^u$, 
The first assertion follows.
For (a) (b), see \cite[Lemma 3.5]{SenTak2}.
\end{proof}


\begin{lemma}\label{globals}
The following holds for each $\omega\in\mathcal P_{n}$:
\begin{itemize}

\item[(a)] $\tau(\omega)\geq 2n$;

\item[(b)] let $\partial^u\omega$ denote any unstable side of $\omega$. Then ${\rm length}(\partial^u\omega)\leq e^{-\lambda n}$;

\item[(c)] if $x\in\omega$, then 
$d_{\rm crit}(f^nx)\geq e^{-10\tau(\omega)}$ for every $0\leq n\leq \tau(\omega)-1.$

\end{itemize}

\end{lemma}
\begin{proof}
(a) follows from \eqref{tau} and $\min\{r(\omega)\colon \omega\in\mathcal P_1\}=2$.
(b) follows from (a) and the fact that
$f^{\tau(\omega)}$ maps $\partial^u\omega$ with uniform expansion as in Lemma \ref{global} 
to a curve in $\tilde\Gamma^u$ of length nearly $1$.
If (c) is not the case, then
$f^{\tau(\omega)}x$ is still close to $Q$, a contradiction.
\end{proof}

\begin{figure}
\begin{center}
\includegraphics[height=4cm,width=10cm]{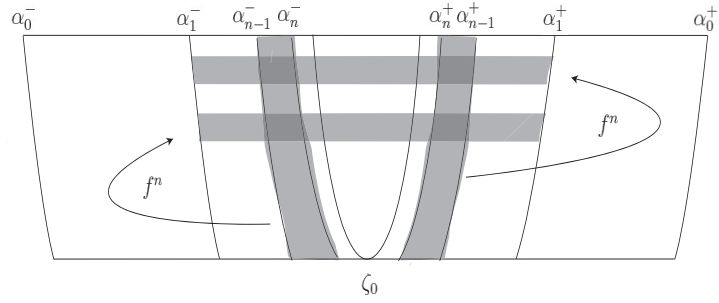}
\caption{The proper rectangles (shaded) in $\mathcal P_1$ with inducing time $n$ and their $f^n$-images}
\end{center}
\end{figure}

\subsection{Rectangles containing points without too deep returns}\label{neg}
We need two lemmas on the recurrence properties of proper rectangles intersecting $G_m$.

\begin{lemma}\label{nonint}
Let $\omega$ be a proper rectangle such that $\omega\cap G_m\neq\emptyset$
for some $m\geq0$. If $\tau(\omega)>m$,
then for any $x\in\omega$,  
$$d_{\rm crit}(f^nx)> b^{\frac{n}{9}}\
\text{ for every }\ m\leq n\leq\tau(\omega)-1.$$

\end{lemma}

\begin{proof}
 Let $m\leq n\leq\tau(\omega)-1$ be such that $f^n\omega\cap I(\delta)\neq\emptyset$.
 Choose $x_0\in\omega\cap G_m$.
 The $f^{n+1}\omega$ is contained in a rectangle whose stable sides are
 two neighboring curves in
 $\{\alpha_k\}_{k>0}$. 
 From the quadratic behavior near the critical points and 
 the exponential convergence of the curves  $\{\alpha_k\}_{k>0}$
 to $\alpha_0$ with exponent $\log 4$,
 for any $x\in\omega$
 we have $2d_{\rm crit}(f^nx)^2> (1/16)2d_{\rm crit}(f^nx_0)^2$.
  This yields
   $d_{\rm crit}(f^nx)> (1/4)d_{\rm crit}(f^nx_0)\geq(1/4)
   b^{\frac{n}{10}}> b^{\frac{n}{9}}$.
\end{proof}

\begin{lemma}\label{behave}
Let $\omega$ be a proper rectangle such that $\omega\cap G_m\neq\emptyset$
for some $m\geq0$. If $\tau(\omega)>m$, then
there exists $k\in[0,m]$ such that the stable sides of $f^k\omega$ are contained in long stable leaves.
\end{lemma}

\begin{proof}
Let $\partial^s\omega$ denote any stable side of $\omega$ and $z$ an endpoint of 
$\partial^s\omega$.
By Lemma \ref{leaf1} and Lemma \ref{sr}, it suffices to show
$d_{\rm crit}(f^nz)> b^{\frac{n}{9}}$ for every $n\geq m.$
Since $\omega\cap G_m\neq\emptyset$,
 this for $m\leq n\leq \tau(\omega)-1$
 follows from Lemma \ref{nonint}.
Since $f^{\tau(\omega)}z\in\alpha_1^-\cup\alpha_1^+$, 
for every $n\geq\tau(\omega)$ we have $f^nz\in f^{n-\tau(\omega)}(\alpha_1^-\cup \alpha_1^+)\subset\alpha_1^+$, and so the desired inequality
for every $n\geq\tau(\omega)$.
\end{proof}

\subsection{Symbolic dynamics}\label{horse}
Let $\mathcal A$ be a finite collection of
proper rectangles contained in the interior of $\Theta$, 
labeled with $1,2,\ldots,\ell=\#\mathcal A$. We assume any two elements of $\mathcal A$ are either disjoint,
or intersect each other only at their stable sides.
Endow $\Sigma_{\ell}=\{1,\ldots,\ell\}^{\mathbb Z}$ with the product topology 
of the discrete topology,
and let $\sigma\colon\Sigma_{\ell}\circlearrowleft$ denote the left shift.
Define a coding map $\pi\colon\Sigma_{\ell}\to\mathbb R^2$ by
$\pi(\{x_i\}_{i\in\mathbb Z})=y$, where
$$\{y\}=\left(\bigcap_{k=1}^{\infty}\omega_k^s\right)
\cap\left(\bigcap_{k=1}^{\infty}\omega_k^u\right)$$ 
and
 $$\omega^s_k=\omega_{x_0}\cap\left(\bigcap_{i=1}^k 
f^{-\tau(\omega_{x_0})}\circ\cdots\circ f^{-\tau(\omega_{x_{i-1}})}\omega_{x_i}\right)\text{ and }
\omega^u_k=\bigcap_{i=1}^{k}f^{\tau(\omega_{x_{-1}})}\circ\cdots\circ f^{\tau(\omega_{x_{-i}})}\omega_{x_{-i}}.$$

\begin{lemma}\label{hyp}
The map $\pi$ is well-defined, continuous, injective, and satisfies $\pi(\Sigma_{\ell})\subset\Omega$.
\end{lemma}

\begin{proof}To show that $\left(\bigcap_{k=1}^{\infty}\omega_k^s\right)
\cap\left(\bigcap_{k=1}^{\infty}\omega_k^u\right)$ is a singleton and so
$\pi_{\mathcal A}$ is well-defined,
it suffices to show that both $\omega^s_k$ and
$\omega^u_k$ get thinner as $k$ increases, and converge to curves intersecting each other exactly at one point.
We argue as follows.


Since $\#\mathcal A$ is finite, the elements of $\mathcal A$ do not accumulate the parabola $f^{-1}\alpha_0^+\cap R$.
By Lemma \ref{behave} there exists $k_0\geq1$ such that 
for each $k\geq k_0$, the stable sides of $F\omega^s_{k}$
are contained in long stable leaves, where $F=f^{\tau(\omega_{x_0})+\cdots+\tau(\omega_{x_{k_0}})+1}$.
By the exponential decrease of the lengths of the unstable sides of this rectangle in $k$, and by
 Lemma \ref{leaf1}(b), these long stable leaves converge as $k\to\infty$ to a single long stable leaf, denoted by $\gamma^s$. It follows that 
$\bigcap_{k=1}^\infty \omega^s_k$ is a curve contained in $F^{-1}\gamma^s$,
joining the two unstable sides of $R$.

The unstable sides of $\omega_k^u$ belong to $\tilde\Gamma^u$.
By \cite[Lemma 2.2]{SenTak1}, the Hausdorff distance between them decreases exponentially in $k$. 
This implies
$\bigcap_{k=1}^{\infty}\omega_k^u\in\Gamma^u$. Hence 
$\left(\bigcap_{k=1}^{\infty}\omega_k^s\right)
\cap\left(\bigcap_{k=1}^{\infty}\omega_k^u\right)\neq\emptyset$ holds.

We have
$$F\left(\left(\bigcap_{k=1}^{\infty}\omega_k^s\right)\cap\left( \bigcap_{k=1}^{\infty}\omega_k^u\right)\right)\subset
F\left(\bigcap_{k=1}^{\infty}\omega_k^s\right)\cap
F\left(\omega_{k_0}^s\cap\bigcap_{k=1}^{\infty}\omega_k^u\right).$$
The first set of the right-hand-side is a subset of $\gamma^s$ and the second is in $\Gamma^u$.
Hence, the set of the left-hand-side is a singleton.
Since $F$ is a diffeomorphism, 
$\left(\bigcap_{k=1}^{\infty}\omega_k^s\right)
\cap\left(\bigcap_{k=1}^{\infty}\omega_k^u\right)$ is a singleton.

 

Since all points outside of $R$ diverges to infinity under positive or negative iteration,
we have $y\in\bigcap_{n\in\mathbb Z}f^nR$, and so $y\in\Omega$ from the first property of the rectangle $R$ in Sect.\ref{family}.
In addition, the above argument shows the continuity of $\pi$.

To show the injectivity, assume $x,y\in\Sigma_{\ell}$, $x\neq y$ and $\pi(x)= \pi(y)$.
Then $\pi(x)$ is contained in the stable side of two neighboring elements of $\mathcal A$.
Hence $f^n \pi(x)$ is not contained in the interior of $\Theta$ for every $n\geq1$, a contradiction.
\end{proof}



\subsection{Bounded distortion}\label{bdd}
We establish distortion bounds
for proper rectangles.
\begin{lemma}\label{lyapbdd}
For every $m\geq0$ there exists a constant $D_m>0$ such that for
any proper rectangle $\omega$ intersecting $G_m$ and $\tau(\omega)>m$,
$$\sup_{x,y\in\Omega\cap\omega}
 \frac{\|D_yf^{\tau(\omega)}|E_y^u\|}{   \|D_xf^{\tau(\omega)}|E_x^u\|  }  \leq D_m.$$
\end{lemma}

\begin{proof}
Let $x,y\in\Omega\cap\omega$.
By the last remark on long unstable leaves in Sect.\ref{stablegeo}, 
$\Omega\cap\omega\subset\mathcal W^u$.
Take a stable side of $\omega$ and denote it by $\partial^s\omega$.
Take $x'\in\partial^s\omega$ (resp. $y'\in\partial^s\omega$) such that $x$ and $x'$ (resp. $y$ and $y'$) lie on the same long unstable leaf.
The Chain Rule gives
\begin{equation*}
\frac{\|D_yf^{\tau(\omega)}|E_y^u\|}{   \|D_xf^{\tau(\omega)}|E_x^u\|  }=
\frac{\|D_{x'}f^{\tau(\omega)}|E_{x'}^u\|}{   \|D_xf^{\tau(\omega)}|E_x^u\|  }\cdot\frac{\|D_{y'}f^{\tau(\omega)}|E_{y'}^u\|}{   \|D_{x'}f^{\tau(\omega)}|E_{x'}^u\|  }\cdot
\frac{\|D_yf^{\tau(\omega)}|E_y^u\|}{   \|D_{y'}f^{\tau(\omega)}|E_{y'}^u\|  }.\end{equation*}
 Lemma \ref{global}(b) bounds the first and the third factors.
For the second one,
by Lemma \ref{behave} there exists $k\in[0,m]$ such that
 $f^k\partial^s\omega$ is contained in a long stable leaf. 
 Then
 \begin{equation*}\frac{\|D_{y'}f^{\tau(\omega)}|E_{y'}^u\|}{   \|D_{x'}f^{\tau(\omega)}|E_{x'}^u\|  }  \leq 
\frac{\|D_{y'}f^{k}|E_{y'}^u\|}{   \|D_{x'}f^{k}|E_{x'}^u\|  }  +
\frac{\|D_{f^ky'}f^{\tau(\omega)-k}|E_{f^ky'}^u\|}{   \|D_{f^kx'}f^{\tau(\omega)-k}|E_{f^kx'}^u\|  } .\end{equation*}
The first term of the right-hand-side is bounded by
 a uniform constant which depends only on $m$ and $f$.
The second one is bounded by Lemma \ref{leaf1}(a).
\end{proof}


\subsection{Approximation of ergodic measures with horseshoes}
Katok established the remarkable result that every hyperbolic measures of differomorphisms
can be in a particular sense approximated by uniformly hyperbolic horseshoes
(See \cite[Theorem S.5.9]{KatHas95} for the precise statement). 
We will need a version of this.
Let $\mathcal M^e(f)$ denote the set of $f$-invariant ergodic Borel probability measures.
\begin{lemma}\label{katok}
Let $\mu\in\mathcal M^e(f)$ satisfy $h(\mu)>0$.
For any $\varepsilon>0$ there exist $q>0$ and a finite collection 
$\mathcal R$ of
proper rectangles such that:

\begin{itemize}
\item[(a)] for each $\omega\in\mathcal R$,
$\tau(\omega)=q$;

\item[(b)] $\left|(1/q)\log \#\mathcal R-h(\mu)\right|<\varepsilon;$

\item[(c)] for any $x\in \bigcup_{\omega\in\mathcal R}\mathcal W^u\cap \omega$, 
$\left|(1/q)\sum_{i=0}^{q-1}\log J^u(f^ix)-\lambda^u(\mu)\right|<\varepsilon$.

\end{itemize}

\end{lemma}

\begin{proof}
By \cite[Theorem S.5.9]{KatHas95}, for any $\varepsilon\in(0,2h(\mu))$ there exists $\nu\in\mathcal M^e(f)$ which is supported on a hyperbolic set and satisfies
$|h(\mu)-h(\nu)|<\varepsilon/2$, 
$|\lambda^u(\mu)-\lambda^u(\nu)|<\varepsilon/3$.
We have $\nu(\Theta)>0$, for otherwise $\nu$ the Dirac measure at $Q$, in contradiction to 
$h(\nu)>0$.

Let $\omega_{\rm S}$ (resp. $\omega_{\rm R}$) denote the connected component of $R\setminus\Theta$ 
at the left (resp. right) of $\zeta_0$, and define
$$\mathcal Q(\nu)=\{\omega\in\mathcal P_1\colon\nu(\omega)>0\}\bigcup\{\omega_{\rm S},\omega_{\rm R}\}.$$
Since $\nu$ is supported on a hyperbolic set, $\#\mathcal Q(\nu)$ is finite. 
We claim that $\mathcal Q(\nu)$ is a generating partition with respect to $\nu$.
Indeed, by \cite[Lemma 3.1]{SenTak1}, there is a continuous surjection
$\iota$ from $\Sigma_2$ to $\Omega$ which gives a symbolic coding of points in $\Omega$.
Since the coding is given by the two rectangles intersecting only at $\zeta_0$,
for any cylinder set $A$ in $\Sigma_2$, $\iota(A)\cap\bigcup\{\omega\colon\omega\in\mathcal Q(\nu)\}$ belongs to the sigma-algebra generated by 
$\bigcup_{n=0}^\infty\bigvee_{i=-n}^n f^{-i}\mathcal Q(\nu)$.
Since cylinder sets form a base of the topology of $\Sigma_2$,
the claim holds.

For $m>0$ let $\Lambda_m$ denote the set of all $x\in\Theta$ for which the following holds:
\begin{itemize}
\item[(i)] $|(1/n)\log\nu(\omega(x))+h(\nu)|<\varepsilon/3$
for every $n\geq m$,
where $\omega(x)$ denotes the element of 
$\bigvee_{i=0}^{n-1}f^{-i}\mathcal Q(\nu)$ containing $x$;

\item[(ii)]  $\left|(1/n)\sum_{i=0}^{n-1}\log J^u (f^ix)-\lambda^u(\nu)\right|\leq\varepsilon/3$
for every $n\geq m$;

\item[(iii)]  $x\in G_m$.
\end{itemize}

\noindent By the Shannon-McMillan-Breimann Theorem,
the Ergodic Theorem and Corollary \ref{G0}, 
$\nu(\Lambda_m)\to\nu(\Theta)$ as $m\to\infty$.
Let 
$$\Lambda_{m,p}=\{x\in\Lambda_m\colon 
f^qx\in\Theta\text{ for some }q\in[p,2p]\}.$$
We claim $\nu(\Lambda_{m,p})\to\nu(\Lambda_m)$ as $p\to\infty$.
To show this, denote by $\chi_\Theta$  the characteristic function of $\Theta$.
Set
$$B_p=\left\{x\in\Lambda_m\colon\frac{1}{p}\sum_{i=0}^{p-1}\chi_\Theta(f^ix)<\frac{5}{4}\nu(\Theta)\text{ and }\frac{1}{2p}\sum_{i=0}^{2p-1}
\chi_\Theta(f^ix)>\frac{5}{8}\nu(\Theta)\right\}.$$
From the Ergodic Theorem,
 $\nu(B_p)\to\nu(\Lambda_m)$ as $p\to\infty$.
 Since $B_p\subset\Lambda_{m,p}$ the claim holds.

Choose $m>0$ such that 
$\nu(\Lambda_m)\geq(1/2)\nu(\Theta)$,
and then choose $p\geq m$ such that
$\nu(\Lambda_{m,p})\geq(1/3)\nu(\Theta)$,
$-(1/p)\log (6p)+(1/p)\log\nu(\Theta)>-\varepsilon/6$ and
$D_m/p<\varepsilon/3$,
where $D_m$ is the constant in Lemma  \ref{lyapbdd}.
For each $q\in[p,2p]$ set $$\Lambda_{m,p,q}=\{x\in \Lambda_{m,p}\colon \min\{n\in[p,2p]\colon f^nx\in\Theta\}=q\}.$$
Choose $q$ such that 
$\nu(\Lambda_{m,p,q})\geq (1/2p)\nu(\Lambda_{m,p})$. 
Define $\mathcal R$ to be the collection of proper rectangles intersecting $\Lambda_{m,p,q}$ 
with inducing time $q$.
Lemma \ref{katok}(a) is immediate from the construction.

Note that elements of $\mathcal R$ are mutually disjoint, altogether cover $\Lambda_{m,p,q}$ and
belong to $\bigvee_{i=0}^{q-1}f^{-i}\mathcal Q(\nu)$.
(i) gives
$\nu(\omega)\leq e^{-q\left(h(\nu)-\frac{\varepsilon}{3}\right)}$ for each $\omega\in\mathcal R$.
Hence $$\#\mathcal R\geq \nu(\Lambda_{m,p,q})e^{q\left((h(\nu)-\frac{\varepsilon}{3}\right)}
\geq\frac{1}{6p}\nu(\Theta)e^{q\left((h(\nu)-\frac{\varepsilon}{3}\right)},$$ and therefore
$$\frac{1}{q}\log\#\mathcal R\geq -\frac{1}{q}\log (6p)+\frac{1}{q}\log\nu(\Theta)+h(\nu)-\frac{\varepsilon}{3}>h(\nu)-\frac{\varepsilon}{2}>h(\mu)-\varepsilon.$$
Similarly we obtain $(1/q)\log\#\mathcal R\leq h(\nu)+\varepsilon/3$.
This proves Lemma \ref{katok}(b).

For each $\omega\in\mathcal R$ choose $x_\omega\in\omega\cap\Lambda_{m,p,q}$ such that
$\left|(1/q)\sum_{i=0}^{q-1}\log J^u(f^ix_\omega)-\lambda^u(\nu)\right|<\varepsilon/3$.
For all $x\in\mathcal W^u\cap\omega$,
\begin{align*}\left|\frac{1}{q}\sum_{i=0}^{q-1}\log J^u(f^ix)-\lambda^u(\mu)\right|\leq&
\left|\frac{1}{q}\sum_{i=0}^{q-1}\log J^u (f^ix)- \frac{1}{q}\sum_{i=0}^{q-1}\log J^u(f^ix_\omega) \right|\\
&+
\left|\frac{1}{q}\sum_{i=0}^{q-1}\log J^u(f^ix_\omega)-\lambda^u(\nu)\right|
+\left|\lambda^u(\nu)-      \lambda^u(\mu)    \right|\\
\leq& \frac{\log D_m}{q}+\frac{\varepsilon}{3}+\frac{\varepsilon}{3}\leq \frac{\log D_m}{p}+\frac{2\varepsilon}{3} <\varepsilon,\end{align*}
where the first term of the right-hand-side of the first inequality is bounded by Lemma \ref{lyapbdd}
and $x_\omega\in G_m$. Hence Lemma \ref{katok}(c) holds.
\end{proof}

\subsection{Construction of a subset of the level set}
The next lemma will be used to construct a subset of each level set with large dimension. 
\begin{prop}\label{lowd} 
Let $\beta\in I$, and
let $\{\mu_n\}_{n=1}^\infty$ be a sequence in $\mathcal M^e(f)$ such that
$h(\mu_n)>0$ and
$\lambda^u(\mu_n)\to\beta$ as $n\to\infty$.
There exists a closed set
$Z\subset \Omega^u(\beta)$ such that
\begin{equation*}
\dim_H^u(Z)\geq \limsup_{n\to\infty}\frac{h(\mu_n)}{\lambda^u(\mu_n)}.\end{equation*}
\end{prop}

 \begin{proof}
 Taking a subsequence if necessary we may assume $|\lambda^u(\mu_n)-\beta|<1/n$
 and $h(\mu_n)/\lambda^u(\mu_n)$ converges.
  We approximate each $\mu_n$ with a horseshoe in the sense of Lemma \ref{katok}, and 
 then construct a set of points which wander around these horseshoes,
 in such a way that their unstable Lyapunov exponents converge to $\beta$.
This is done along the line of \cite{Chu10}.

  By Lemma \ref{katok}, for each $n$
there exist $q_n>0$ and a family $\mathcal R_n$ of proper rectangles
such that
$\tau(\omega)=q_n$  for each $\omega\in\mathcal R_n$
and

\begin{equation}\label{00}
\frac{1}{q_n}\log \#\mathcal R_n\geq h(\mu_n)-\frac{1}{n};\end{equation}

\begin{equation}\label{11}
\sup\left\{\left|\frac{1}{q_n}\sum_{j=0}^{q_n-1}\log J^u(f^jx)-\lambda^u(\mu_n)\right|\colon\
x\in \bigcup_{\omega\in\mathcal R_n} \mathcal W^u\cap\omega\right\}<\frac{1}{n}.
\end{equation}

For an integer $\kappa\geq1$ let
$$\mathcal R_n(\kappa)=\{\omega_0\cap f^{-q_n}\omega_1\cap\cdots\cap f^{-(\kappa-1)q_n}\omega_{\kappa-1}
\colon\omega_1,\ldots,\omega_{\kappa-1}\in\mathcal R_n\}.$$
Elements of $\mathcal R_{n}(\kappa)$ are proper rectangles with inducing time $\kappa q_n$,
and $\#\mathcal R_{n}(\kappa)=(\#\mathcal R_n)^{\kappa}$ holds.

Let $\{\kappa_n\}_{n=1}^\infty$ be a sequence of positive integers.
For each $k\geq1$ let $(N,s)=(N(k),s(k))$ be a pair of integers such that
$$
k=\kappa_1+\kappa_2+\cdots+\kappa_{N-1}+s \ \text{and}\ \ 0\leq s< \kappa_{N}.$$
Define $\mathcal S(k)$ to be the collection of proper rectangles of the form
$$\omega_0\cap f^{-\kappa_1q_1}\omega_1\cap
\cdots\cap f^{-\kappa_1q_1-\cdots-\kappa_{N-1}q_{N-1}}\omega_{N},$$
where $\omega_n\in\mathcal R_{n}(\kappa_{n+1})$ $(n=0,\ldots,N-1)$ and $\omega_{N}\in\mathcal R_{N}(s).$
Elements of $\mathcal S(k)$ are proper rectangles with inducing time $\kappa_1q_1+\cdots+\kappa_{N-1}q_{N-1}+sq_N$.
The set $\bigcup_{\omega\in\mathcal S(k)}\omega$ is compact, and
decreasing in $k$.

Let $\gamma^u(\zeta_0)$ denote the unstable side of $\Theta$ containing $\zeta_0$.
Set
$$Z=\gamma^u(\zeta_0)\cap \bigcap_{k=1}^\infty\bigcup_{\omega\in\mathcal S(k)}\omega.$$
We show $Z\subset\Omega^u(\beta)$. 
Let $x\in Z$.
For each large integer $M\geq\kappa_1q_1$, choose $(N,s)$ such that
$0\leq s<\kappa_{N}$ and
 $0\leq M-(\kappa_1q_1+\cdots+\kappa_{N-1}q_{N-1}+sq_{N})<q_{N}$.
 The triangle inequality gives
 $$\left|\sum_{j=0}^{M-1}\log J^u(f^jx)-M\beta\right|\leq I+I\!I+I\!I\!I+I\!V,$$
 where
 \begin{align*}
 I&=\sum_{j=0}^{\kappa_1-1}\left|\sum_{l=0}^{q_1-1}\log J^u(f^{q_1j+l}x)-q_{1}\beta\right|;\\
 I\!I&=\sum_{n=1}^{N-1}\sum_{j=0}^{\kappa_{n}-1}\left|\sum_{l=0}^{q_{n}-1}\log J^u(f^{\kappa_1q_1+\cdots+\kappa_{n-1}q_{n-1}+jq_{n}+l}x)-q_{n}\beta\right|;\\
 I\!I\!I&=\sum_{j=0}^{s-1}\left|\sum_{l=0}^{q_N-1}\log J^u(f^{\kappa_1q_1+\cdots+\kappa_{N-1}q_{N-1}+jq_{N}+l}x)-q_{N}\beta\right|;\\
 I\!V&=\left|\sum_{l=0}^{M-(\kappa_1q_1+\cdots+\kappa_{N-1}q_{N-1}+sq_{N})-1}\log J^u(f^{\kappa_1q_1+\cdots+\kappa_{N-1}q_{N-1}+sq_{N}+l}x)-
(M-(\kappa_1q_1+\cdots+\kappa_{N-1}q_{N-1}+sq_{N}))\beta\right|.
 \end{align*}
Using \eqref{11},
$$\left|\sum_{l=0}^{q_1-1}\log J^u(f^{jq_1+l}x)-q_1\beta\right|
\leq\left|\sum_{l=0}^{q_1-1}\log J^u(f^{jq_1+l}x)-q_{1}\lambda^u(\mu_1)\right|+
\left|q_1\lambda^u(\mu_1)-q_{1}\beta\right|\leq 2q_1,$$
and similarly
\begin{align*}
\left|\sum_{l=0}^{q_{n}-1}\log J^u(f^{\kappa_1q_1+\cdots+\kappa_{n-1}q_{n-1}+jq_{n}+l}x)-q_{n}\beta\right|\leq\frac{2q_{n}}{n}.
\end{align*}
Summing these and other reminder terms we get
\begin{align*}
\left|\sum_{j=0}^{M-1}\log J^u(f^jx)-M\beta\right|&\leq\sum_{n=1}^{N-1}
\frac{2q_{n}
\kappa_{n}}{n}+\frac{2q_{N}s}{N}+(M-(\kappa_1q_1+\cdots+\kappa_{N-1}q_{N-1}+sq_{N}  ))
(\log 5-\beta)\\
&\leq \frac{3q_{N-1}
\kappa_{N-1}}{N}+\frac{2q_{N}s}{N}+q_N(\log 5-\beta)\leq\frac{4M}{N},
\end{align*}
where the second and the last inequalities hold provided
$\kappa_{N-1}$ is sufficiently large compared to 
$q_1,q_2,\ldots,q_{N},\kappa_1,\kappa_2,\ldots,\kappa_{N-2}.$
Since $N\to\infty$ as $M\to\infty$, we get $\lambda^u(x)=\beta$.

For each $k$ and
 $\omega\in\mathcal S(k)$ 
choose a point $x_\omega\in \omega\cap Z$, and define an atomic probability measure $\nu_k$ equally 
distributed on the set $\{x_\omega\colon \omega\in\mathcal S(k)\}$.
Let $\nu$ denote an accumulation point of the sequence $\{\nu_k\}_k$. Since $Z$ is closed,
$\nu(Z)=1$.
For  $\varepsilon>0$ and $x\in W^u$ let $D_\varepsilon(x)$ denote the closed ball in $W^u$ of radius $\varepsilon$ about $x$. 
By virtue of \cite[Lemma 2.1]{You82},
the desired lower estimate in Lemma \ref{lowd} 
follows if 
\begin{equation}\label{mass}
\liminf_{\varepsilon\to 0}\frac{\log \nu D_\varepsilon(x)}{\log \varepsilon} 
\geq \limsup_{n\to\infty}\frac{h(\mu_n)}{\lambda^u(\mu_n)}\quad \forall x\in Z.
\end{equation}


To show \eqref{mass} consider the set of
pairs $(n,s)$ of integers such that
$n>1$ and $0\le s< \kappa_{n}$.
We introduce an order in this set as follows:
$(n_1,s_1)<(n_2,s_2)$ if $n_1<n_2$, or $n_1=n_2$ and $s_1<s_2$.
For a pair $(n,s)$ in this set, define
\begin{equation*}
a_{n,s}= \exp \left[-\kappa_{n-1}q_{n-1}\left(\lambda^u(\mu_{n-1})+ \frac{2}{n-1}\right)
-sq_{n}\left(\lambda^u(\mu_{n})+ \frac{1}{n}\right) \right].\end{equation*}
We have
\[
a_{n,0}= \exp \left(-\kappa_{n-1}q_{n-1}\left(\lambda^u(\mu_{n-1})+ \frac{2}{n-1}\right)\right),\]
and
\[
a_{n-1,\kappa_{n}-1}= \exp \left(-\kappa_{n-2}q_{n-2}\left(\lambda^u(\mu_{n-2})+
\frac{2}{n-2}\right)
-(\kappa_{n-1}-1)q_{n-1}\left(\lambda^u(\mu_{n-1})+ \frac{1}{n-1}\right) \right).\]
From Using the uniform boundedness of $\{\lambda^u(\mu_n)\}_n$ 
We choose $\{\kappa_n\}_n$ so that $\kappa_{n-1}q_{n-1}\gg \kappa_{n-1}q_{n-2}$ and as a result 
$a_{n,0}<a_{n-1,\kappa_{n}-1}$,
namely, the sequence $\{a_{n,s}\}_{(n,s)}$ is monotone decreasing.

For sufficiently small $\varepsilon>0$ set $k(\varepsilon)=\max\{ k\geq1\colon\varepsilon\leq a_{N(k),s(k)}\}$, and define $N=N(k(\varepsilon))$,
$s=s(k(\varepsilon))$.
For each $\omega\in\mathcal S(k)$ set $\omega^u=\omega\cap\gamma^u(\zeta_0).$  
From \eqref{00}, for any $y\in\omega^u$ we have
\begin{align*}
\left|\sum_{j=0}^{\kappa_1q_1+\cdots+\kappa_{N-1}q_{N-1}+
sq_{N}-1}\log J^u(f^jy)\right|&\leq \kappa_{N-1}q_{N-1}\left(\lambda^u(\mu_{N-1})+\frac{2}{N-1}\right)+s q_{N}
\left(\lambda^u(\mu_{N})+\frac{1}{N}\right).
\end{align*}
where the second and the last inequalities hold provided
$\kappa_{N-1}$ is sufficiently large compared to 
$q_1,q_2,\ldots,q_{N},\kappa_1,\kappa_2,\ldots,\kappa_{N-2}.$

Since the curve $f^{\kappa_1q_1+\cdots+\kappa_{N-1}q_{N-1}+sq_N}\omega^u$ belongs to $\tilde\Gamma^u$,
the Mean Value Theorem gives
\begin{equation}\label{Q}
{\rm length}(\omega^u)\geq\frac{1}{2} \exp\left[-\kappa_{N-1}q_{N-1}\left(\lambda^u(\mu_{N-1})+\frac{2}{N-1}\right)-sq_N\left(\lambda^u(\mu_{N})+\frac{1}{N}\right)\right].
\end{equation}
Hence, for any $x\in Z$  the number of elements of $\mathcal S(k)$ which intersect
$D_\varepsilon(x)$ is at most
$$\frac{2\varepsilon}{\inf_{\omega^u}{\rm length}(\omega^u)}\leq \frac{2a_{N,s}}{\inf_{\omega^u}{\rm length}(\omega^u)}\leq 4.$$

By construction, for every $p\geq k$,
$$\nu_p(\omega^u)=\frac{\#\{\omega'\in \mathcal S(p)\colon \omega'\subset \omega\}}{\#\mathcal S(p)}=\frac{1}{\#\mathcal S(k)}.$$
Since $\nu$ charges no weight to the endpoints of $\omega^u$, 
\begin{equation*}\label{nu0}\nu(\omega^u)=\lim_{p\to\infty}\nu_p(\omega^u)=\frac{1}{\#\mathcal S(k)}.\end{equation*}
Using this and \eqref{00},
\begin{align*}
\nu D_\varepsilon(x)&\leq\frac{4}{\#\mathcal S(k)} \leq\frac{4}{   (\#\mathcal R_{N-1})^{\kappa_{N-1}}\cdot(\#\mathcal R_{N})^{s}    }\\
   &\leq4\exp\left[-\kappa_{N-1}q_{N-1}\left(h(\mu_{N-1})-\frac{1}{N-1}\right)-
sq_{N}\left(h(\mu_{N})-\frac{1}{N}\right)\right].\end{align*}
This yields
$$\frac{\log\nu D_\varepsilon(x)}{\log\varepsilon}\geq\frac{\kappa_{N-1}q_{N-1}\left(h(\mu_{N-1})-
1/(N-1)\right)+sq_{N}\left(h(\mu_{N})-1/N\right)}
{\kappa_{N-1}q_{N-1}\left(\lambda^u(\mu_{N-1})+2/(N-1)\right)
+sq_{N}\left(\lambda^u(\mu_{N})+1/N\right)}+\frac{\log4}{\log\varepsilon}.$$
The desired inequality holds since $N \to\infty$ as $\varepsilon\to0$.
This completes the proof of Lemma \ref{lowd}.
\end{proof}

\subsection{Approximation with measures with positive entropy}\label{nergodic}
We need two approximation lemmas on measures. The first one asserts that
for any ergodic measure with zero entropy one can find another ergodic one with small positive entropy
and similar unstable Lyapunov exponent.
The second one asserts that for any non ergodic measure one can find an ergodic one with similar
entropy and similar unstable Lyapunov exponent. 

\begin{lemma}\label{apr}
For any $\mu\in\mathcal M^e(f)$ with $h(\mu)=0$ and $\varepsilon>0$
there exists $\nu\in\mathcal M^e(f)$ such that $0<h(\nu)<\varepsilon$
and $|\lambda^u(\mu)-\lambda^u(\nu)|<\varepsilon$.
\end{lemma}
\begin{proof}
By Katok's Closing Lemma \cite[Main Lemma]{Kat80} 
there exists a periodic point $p$ and an atomic measure $\mu'$ supported
on the orbit of $p$ such that
 $|\lambda^u(\mu)-\lambda^u(\mu')|<\varepsilon/2$.
 Since there is a transverse homoclinic point associated to $p$, 
from the Poincar\'e-Birkhoff-Smale Theorem (see e.g. \cite[Theorem 6.5.5]{KatHas95}) 
there exists a non trivial basic set containing $p$ and the transverse homoclinic point.
The isolating neighborhood of the basic set is a thin strip around the stable manifold of $p$.
Taking a sufficiently thin isolating neighborhood one can make sure
that the measure of maximal entropy of $f$ restricted to the basic set, denoted by $\nu$,
satisfies  $0<h(\nu)<\varepsilon$ and $|\lambda^u(\mu')-\lambda^u(\nu)|<\varepsilon/2$.
\end{proof}

\begin{lemma}\label{approximate}
For any $\mu\in\mathcal M(f)$ and $\varepsilon>0$
there exists $\nu\in\mathcal M^e(f)$ such that
 $h(\nu)>0$, $|h(\mu)- h(\nu)|<\varepsilon$ and $|\lambda^u(\mu)-\lambda^u(\nu)|<\varepsilon$.
 \end{lemma}
 
 \begin{proof}

Considering the ergodic decomposition of $\mu$ one can find a linear combination
$\mu'=a_1\mu_1+\cdots+a_s\mu_s$ of ergodic measures such that 
 $|h(\mu)-h(\mu')|<\varepsilon/2$ and $|\lambda^u(\mu)-\lambda^u(\mu')|<\varepsilon/2$.
 By Lemma \ref{apr}, for each $\mu_i$ there exists $\nu_i\in\mathcal M^e(f)$ such that
 $h(\nu_i)>0$, $|h(\mu_i)-h(\nu_i)|<\varepsilon/2$ and $|\lambda^u(\mu_i)-\lambda^u(\nu_i)|<\varepsilon/2$.
 Set $\nu=a_1\nu_1+\cdots+a_s\nu_s.$
Then $h(\nu)>0$,
  $|h(\mu')-h(\nu)|<\varepsilon/2$ and $|\lambda^u(\mu')-\lambda^u(\nu)|<\varepsilon/2.$
Hence $|h(\mu)-h(\nu)|<\varepsilon$ and $|\lambda^u(\mu)-\lambda^u(\nu)|<\varepsilon.$ 

We note that 
$f|\Omega$ is a factor of the full shift on two symbols
 \cite[Lemma 3.1]{SenTak2}, and therefore has the specification property
 \cite[Lemma 1(b)]{Sig74}.
Hence, ergodic measures are entropy-dense
\cite{EizKifWei94}:
there exists a sequence
$\{\xi_n\}_n$ in $\mathcal M^e(f)$ such that $\xi_n\to\nu$ and $h(\xi_n)\to h(\nu)$ as $n\to\infty$. 
By \cite[Lemma 4.4]{SenTak1} and $\nu\{Q\}=0$, we obtain  $\lambda^u(\xi_n)\to\lambda^u(\nu)$. 
 \end{proof}

\section{Proofs of the theorems}
In this section we bring the results in Sect.2 together and prove the theorems.
In Sect.\ref{completeness} we prove Theorem A.
 In Sect.\ref{upperest} we complete the proof of Theorem B.
In Sect.\ref{continuous} we prove Theorem C.
In Sect.\ref{irregular} we prove Theorem D.

\subsection{Domain of the Lyapunov spectrum}\label{completeness}
We now prove Theorem A.
\medskip

\noindent{\it Proof of Theorem A.}
Let $\beta\in I$. For $\varepsilon>0$
set \begin{equation}\label{Feps}d^u_\varepsilon={\sup}\left\{\frac{h(\mu)}{\lambda^u
(\mu)}\colon\mu\in\mathcal M(f),\ \left|\lambda^u(\mu)
-\beta\right|<\varepsilon\right\}.\end{equation}
We also define $d^{u,e}_\varepsilon$ by restricting the range of the supremum to
the set
$\mathcal M^e(f)$ of ergodic measures.
The next lemma establishes the ``if" part of Theorem A.

\begin{lemma}\label{complete}
For any $\beta\in I$, $\Omega^u(\beta)\neq\emptyset$ and
$L^u(\beta)\geq \displaystyle{\lim_{\varepsilon\to0}d^{u,e}_\varepsilon}.$
In addition, if $\beta\in{\rm int}I$, then $L^u(\beta)>0$.
 \end{lemma}
\begin{proof}
In the case $\beta\in  {\rm int}I$,
by Lemma \ref{approximate} it is possible to choose $\mu_1$, $\mu_2\in\mathcal M^e(f)$ 
with positive entropy and satisfying $\lambda^u(\mu_1)<\beta<\lambda^u(\mu_2)$.
Choose $t\in(0,1)$ such that $t\lambda^u(\mu_1)+(1-t)\lambda^u(\mu_2)=\beta$.
By Lemma \ref{approximate} again, there exists a sequence $\{\nu_n\}_n$ in $\mathcal M^e(f)$
with $\displaystyle{\lim_{n\to\infty}}h(\nu_n)>0$ and
$\lambda^u(\nu_n)\to\beta$ as $n\to\infty$.
Lemma \ref{lowd} yields
$\Omega^u(\beta)\neq\emptyset$ and 
$L^u(\beta)\geq  \displaystyle{\lim_{\varepsilon\to0}d^{u,e}_\varepsilon}>0.$
In the case $\beta=\lambda_m^u$, by Lemma \ref{approximate}
it is possible to choose a sequence $\{\mu_n\}_n$ in $\mathcal M^e(f)$ such that $\lambda^u(\mu_n)\to\lambda_m^u$ as $n\to\infty$
and $h(\mu_n)>0$ for every $n$.
Lemma \ref{lowd} yields $\Omega^u(\beta)\neq\emptyset$ and 
$L^u(\beta)\geq \displaystyle{\lim_{\varepsilon\to0}d^{u,e}_\varepsilon}.$
A proof for the case $\beta=\lambda_M^u$ is completely analogous. 
\end{proof}

For a proof of the ``only if" part in Theorem A we need a couple of lemmas.

\begin{lemma}\label{new2}
If $x\in \bigcup_{m=0}^\infty G_m\setminus W^s(Q)$, then $\bar{\lambda}^u(x)\geq\lambda_m^u$.
\end{lemma}
\begin{proof}

 Let
$x\in G_m$. 
Since $f^nx\in\Theta$ holds for
infinitely many $n>0$, there exists an infinite nested sequence $\omega_0\supset\omega_1\supset\cdots$ of proper rectangles containing $x$.
From Lemma \ref{hyp},
each $\omega_n$ contains a periodic point of period $\tau(\omega_n)$, denoted by $q_n$.
Since $\omega_n\cap G_m\neq\emptyset$,
Lemma \ref{lyapbdd} gives
$$\left|\frac{1}{\tau(\omega_n)}\sum_{i=0}^{\tau(\omega_n)-1}
\log J^u(f^iq_n)-\log J^u(f^ix)\right|\leq \frac{\log D_m}{\tau(\omega_n)}.$$
Since
 $\tau(\omega_n)\to\infty$ as $n\to\infty$, the desired inequality follows.
 \end{proof}

The next upper semi-continuity result
 follows from a slight modification the proof of \cite[Lemma 4.3]{SenTak1} in which a convergent sequence 
of $f$-invariant measures were treated.
For $x\in\Omega$ and $n\geq1$
write $\delta_x^n=(1/n)\sum_{i=0}^{n-1}\delta_{f^ix}$, where
$\delta_{f^ix}$ denotes the Dirac measure at $f^ix$.

\begin{lemma}\label{sentak}
Let $x\in\Omega$ and $\{n_k\}_k$, $n_k\nearrow\infty$ be such that $\delta_{x}^{n_k}$
converges
weakly to $\mu\in\mathcal M(f)$. Then $$\limsup_{k\to\infty}\int\log J^u\delta_x^{n_k}\leq\lambda^u(\mu).$$
\end{lemma}
\begin{proof}
If $x\in W^s(Q)$, then $\delta_x^{n_k}\to\delta_Q$ and 
$\int\log J^u\delta_x^{n_k}\to\lambda^u(\delta_Q)$ as $k\to\infty$,
and so the desired inequality holds.
Assume $x\notin W^s(Q)$.
Write $\mu=u\delta_Q+(1-u)\nu$, $0\leq u\leq 1$, $\nu\in\mathcal M(f)$
 and $\nu\{Q\}=0$.  
Let $\varepsilon>0$. Let $V$ be a small open set containing $Q$,
$\mu(\partial V)=0$ and $\mu(V)\leq u+\varepsilon$.
 Fix a partition of unity $\{\rho_{0},\rho_{1}\}$ on $R$ such that
${\rm supp}(\rho_{0})=\overline{\{x\in R\colon \rho_{0}(x)\neq0\}}\subset V$
and $Q\notin {\rm supp}(\rho_{1})$.
Hence
$$\lim_{k\to\infty}\frac{1}{n_k}\{0\leq i<n_k\colon f^ix\in V\}
=\lim_{k\to\infty}\delta_x^{n_k}(V)=\mu(V)
\leq u+\varepsilon.$$ Since $x\notin W^s(Q)$, the forward orbit of $x$ is a concatenation of segments in $V$ and those out of $V$
Let $l_k$ denote the number of segments in $V$ up to time $n_k$.
If $0\leq i_1<i_2$ are such that
$f^{i_1}x\notin V$, $f^ix\in V$ for $i=i_1+1,\ldots,i_2-1$ and $f^{i_2}x\notin V$, then
$\|D_{f^{i_1}x}f^{i_2-i_1}|E^u_{f^{i_1}x}\|\leq Ce^{\lambda^u(\delta_Q)(i_2-i_1)}$.
Then 
$$\int\rho_{0}\log J^ud\delta_x^{n_k}=
\frac{1}{n_k}\sum_{i=0}^{n_k-1} (\rho_{0}\log
J^u)\circ f^i(x)\leq (u+2\varepsilon)\lambda^u(\delta_Q)+C\frac{l_k}{n_k}.$$
If $u<1$, then the weak convergence for the sequence $\{\frac{\delta_x^{n_k}-u\delta_Q}{1-u}\}_k$ of measures implies
$$\lim_{n\to\infty}\int\rho_{1}\log J^ud\delta_x^{n_k}=(1-u)\int\rho_{1}\log J^ud\nu\leq(1-u)\lambda^u(\nu).$$
The same inequality remains to hold for the case $u=1$.
Hence we have
\begin{align*}\limsup_{k\to\infty}\int \log J^ud\delta_x^{n_k}&\leq
\limsup_{k\to\infty}\int\rho_{0}\log
J^ud\delta_x^{n_k}+\lim_{k\to\infty}
\int\rho_{1}\log
J^ud\delta_x^{n_k}\\
&\leq (u+2\varepsilon)\lambda^u(\delta_Q)+
C\cdot\limsup_{k\to\infty}\frac{l_k}{n_k}
+(1-u)\lambda^u(\nu).
\end{align*}
The second term can be made arbitrarily small
by shrinking $V$.
 Then letting $\varepsilon\to0$ yields
the desired inequality. 
\end{proof}

To finish the proof of Theorem A, recall that $\hat\Omega^u=\{x\in\Omega^u\colon \underline{\lambda}^u(x)\neq \bar{\lambda}^u(x) \}$.
Let $x\in\Omega^u\setminus\hat\Omega^u$ and suppose $\lambda^u(x)\neq\lambda^u(\zeta_0)$.
It suffices to show $\lambda^u(x)\in I$.
Lemma \ref{new} gives
$x\in\bigcup_{m=0}^\infty G_m$.
If $x\in W^s(Q)$, then $x=Q$ and so $\lambda^u(x)=\lambda^u(Q)\in I$.
 Otherwise,
 Lemma \ref{new2} gives
$\lambda^u(x)\geq\lambda_m^u$. 
Since $\Omega$ is compact, 
there  is a subsequence $\{n_k\}_k$, $n_k\nearrow\infty$ such that 
 $\delta_x^{n_k}\to\mu\in\mathcal M(f)$
and $\displaystyle{\limsup_{k\to\infty}}\int\log J^u d\delta_x^{n_k}=\lambda^u(x)$.
Lemma \ref{sentak} gives $\lambda^u(x)\leq \lambda^u(\mu)\leq\lambda^u_M$.
\qed

\subsection{Formula for the Lyapunov spectrum}\label{upperest}
We now prove Theorem B.

\medskip
\noindent{\it Proof of Theorem B.}
We argue in two steps. Let $\beta\in I$. In Step 1 we estimate $L^u(\beta)$ from below.
In Step 2 we estimate $L^u(\beta)$ from above.
\medskip

\noindent{\it Step1(Lower estimate).} 
Let $\mu\in\mathcal M(f)$ be non ergodic with $h(\mu)>0$.
By Lemma \ref{approximate}, for any $\varepsilon>0$
there exists $\nu\in\mathcal M^e(f)$
such that
$|h(\mu)-h(\nu)|<\varepsilon$ and $|\lambda^u(\mu)-\lambda^u(\nu)|<\varepsilon$.
Since $h(\mu)\leq \log2$ and $\lambda^u(\mu)<\log5$,
$$\left|\frac{h(\mu)}{\lambda^u(\mu)}-\frac{h(\nu)}{\lambda^u(\nu)}\right|<\frac{(\log 2+\log 5)\varepsilon}{(\lambda_m^u)^2}<\frac{3\varepsilon}{(\lambda_m^u)^2}.$$
It follows that
$$d^{u,e}(2\varepsilon)>
d^u_\varepsilon-\frac{3\varepsilon}{(\lambda_m^u)^2}.$$ We obtain
$\displaystyle{\lim_{\varepsilon\to0}d^{u,e}_\varepsilon\geq\lim_{\varepsilon\to0}d^u_\varepsilon}.$
From this and  Lemma \ref{complete},
$L^u(\beta)\geq \displaystyle{\lim_{\varepsilon\to0}d^u_\varepsilon}$
follows.
\medskip



\noindent{\it Step2(Upper estimate).} 
From Lemma \ref{new}, the unstable Lyapunov exponents are undefined for points in $\Omega_*$. Hence 
$$\Omega^u(\beta)=\bigcup_{m=0}^\infty \Omega^u(\beta)\cap G_m.$$
From the next Lemma and the countable stability of $\dim_H^u$,
we obtain $L^u(\beta)\leq \displaystyle{\lim_{\varepsilon\to0}d^u_\varepsilon}$.

\begin{prop}\label{c}
 For any $\beta\in I$ and every  $m\geq0$,
$\dim_H^u(\Omega^u(\beta)\cap G_m)\leq \displaystyle{\lim_{\varepsilon\to0}d^u_\varepsilon}.$
\end{prop}

\noindent{\it Proof of Lemma \ref{c}.}
Recall that $\gamma^u(\zeta_0)$ is the unstable side of $\Theta$ containing $\zeta_0$.
Set
$$\tilde\Omega^u(\beta)=\{x\in\Omega^u(\beta)\cap \gamma^u(\zeta_0)\colon f^nx\in\Theta\text{ for infinitely many $n>0$}\}.$$
Since $\gamma^u(\zeta_0)$ contains a fundamental domain in $W^u$,
for any $x\in \Omega^u(\beta)$ which is not the fixed point in $W^u$ there exists $n\in\mathbb Z$ such that
$f^nx\in\gamma^u(\zeta_0)$. 
From the countable stability and the $f$-invariance
of $\dim_H^u$,
$L^u(\beta)=
\dim_H ^u(\Omega^u(\beta)\cap\gamma^u(\zeta_0))$.
Since points in $\Omega^u(\beta)\cap\gamma^u(\zeta_0)$ which return to $\Theta$ under forward iteration
only finitely many times form a countable subset,
we have $L^u(\beta)=\dim_H^u (\tilde\Omega^u(\beta))$.


From this point on,
we restrict ourselves to $\tilde\Omega^u(\beta)$.
 For $c>0$ let $D_c(\zeta_0)$ denote the closed ball in $W^u$ of
  radius $r$ about $\zeta_0$. Define
$$\mathcal A_{n,\varepsilon}=\left\{
 \omega\in\mathcal P_{n}\colon \omega\cap G_m\neq\emptyset,\ \omega\cap
 D_c(\zeta_0)=\emptyset,\
\inf_{x\in\omega\cap\gamma^u(\zeta_0)}\left|\frac{1}{\tau(\omega)}\sum_{i=0}^{\tau(\omega)-1}\log J^u(f^ix)-\beta\right|<\frac{\varepsilon}{2}\right\}.$$
Observe that $\mathcal A_{n,\varepsilon}$ is a finite set, because its elements 
do not intersect $D_c(\zeta_0)$.
For each $\omega\in\mathcal A_{n,\varepsilon}$
write $\omega^u=\omega\cap\gamma^u(\zeta_0)$ and set
$\mathcal A^u_{n,\varepsilon}=\{\omega^u\colon\omega\in\mathcal A_{n,\varepsilon}\}$.
Clearly we have
$$(\tilde\Omega^u(\beta)\cap G_m)\setminus D_c(\zeta_0)\subset\limsup_{n\to\infty}\bigcup_{
\omega^u\in\mathcal A_{n,\varepsilon}^u}\omega^u.$$
It is enough to show
\begin{equation}\label{pperd}
\limsup_{n\to\infty}\frac{1}{n}\log\sum_{\omega^u\in\mathcal A^u_{n,\varepsilon}}{\rm length}(\omega^u)^{ d^u_\varepsilon}\leq 0\ \text{ for any $\varepsilon>0$.}
\end{equation}
Indeed, if this holds, then
using ${\rm length}(\omega^u)\leq e^{-\lambda n}$ from Lemma \ref{globals}(b),
for any $\rho>0$ we have
$$\limsup_{n\to\infty}\frac{1}{n}\log\sum_{A\in\mathcal A^u_{n,\varepsilon}}{\rm length}(\omega^u)^{ d^u_\varepsilon+\rho}\leq-\lambda\rho.$$
It follows 
that $\sum_{A\in\mathcal A^u_{n,\varepsilon}}{\rm length}(\omega^u)^{ d^u_\varepsilon+\rho}$
has a negative growth rate as $n$ increases.
Therefore
the Hausdorff $( d^u_\varepsilon+\rho)$-measure of the set $(\tilde\Omega^u(\beta)\cap G_m)\setminus
D_c(\zeta_0)$
is $0$. Since $\rho>0$ is arbitrary,
$\dim_H^u((\tilde\Omega^u(\beta)\cap G_m)\setminus D_c(\zeta_0))\leq d^u_\varepsilon$, and 
by the countable stability of $\dim_H^u$ we obtain
$\dim_H^u(\tilde\Omega^u(\beta)\cap G_m)\leq d^u_\varepsilon$.
 Letting
$\varepsilon\to0$ yields the desired inequality in Lemma \ref{c}. 
\medskip

It is left to prove \eqref{pperd}.
Set $\ell=\#\mathcal A_{n,\varepsilon}$ and
Write
$\mathcal A_{n,\varepsilon}=\{\omega{(1)},\omega{(2)},\ldots,\omega{(\ell)}\}$
so that 
\begin{equation}\label{align}
\tau(\omega{(1)})\geq\tau(\omega{(s)})>m\ \text{ for every }s\in \{1,2,\ldots,t\}.\end{equation}
Let $\pi_{\ell}\colon\Sigma_\ell\to\bigcup_{\omega\in\mathcal A_{n,\varepsilon}}\omega$ denote the coding map defined in Sect.\ref{horse}
and  $\sigma\colon\Sigma_\ell\circlearrowleft$ the left shift.
Define 
$$B=\{\underline{a}\in\Sigma_\ell\colon\pi\underline{a}\subset W^s(P)\setminus\{P\}\}.$$
Proper rectangles can intersect each other only at their stable sides, and there is only one proper rectangle containing $P$
in its stable side. Hence,
for any $\underline{a}\in \Sigma_\ell\setminus B$ there exists a unique element 
of $\mathcal A_{n,\varepsilon}$ containing $\pi\underline{a}$ which we denote by
$\omega(\underline{a})$.
Define 
$\Phi\colon \Sigma_\ell\setminus B\to\mathbb R$ by
$$\Phi(\underline{a})=- d^u_\varepsilon\sum_{i=0}^{\tau(\omega(\underline{a}))-1}\log J^u(f^i(\pi\underline{a})).$$
Since $\pi(\Sigma_\ell)\subset\Omega\setminus\{Q\}$ and $\log J^u$ is continuous except at $Q$,
$\Phi$ is continuous.

Let
 $\mathcal M(\sigma)$ denote the space of $\sigma$-invariant Borel probability measures on $\Sigma_\ell$
 endowed with the topology of weak convergence.
For each $k\geq 1$ define an atomic probability measure $\nu_k\in\mathcal M(\sigma)$ concentrated on 
the set $E_k=\{\underline{a}\in \Sigma_\ell\colon \sigma^{k}\underline{a}=\underline{a}\}$ by
$$\nu_k=  
\left(\sum_{\underline{b}\in E_k}\exp\left(     S_{k}\Phi(\underline{b})            \right)\right)^{-1}\sum_{\underline{a}\in E_k}
\exp\left(S_{k}\Phi(\underline{a})\right)
\delta_{\underline{a}},$$
where $S_{k}\Phi=\sum_{i=0}^{k-1}\Phi\circ\sigma^i$ and
$\delta_{\underline{a}}$ denotes the Dirac measure at $\underline{a}$.
Let $\nu_0$ denote an accumulation point of the 
sequence $\{\nu_k\}_k$ in $\mathcal M(\sigma)$. Taking
a subsequence if necessary we may assume $\nu_k\to\nu_0$.
We have $\nu_0\in\mathcal M(\sigma)$.

\begin{sublemma}\label{bounded}
For any $\nu\in\mathcal M(\sigma)$, $\nu(B)=0$.
\end{sublemma}
\begin{proof}
If $\nu(B)>0,$ then since $\pi(B)\subset W^s(P)\setminus\{P\}$ one can choose a set $A\subset B$
such that $\nu(A)>0$ and $\pi(A)\cap \pi(\sigma^nA)=\emptyset$ for every $n>0$.
Since $\nu(\sigma ^nA)=\nu(A)$, $\nu$ cannot be a probability, a contradiction.
\end{proof}

Define a Borel probability measure $\overline{\mu}$ on $\pi(\Sigma_\ell)$ by 
$$\overline{\mu}=\sum_{\omega\in\mathcal A_{n,\varepsilon}}\nu_0|_{\pi^{-1}\omega}.$$
By Sublemma \ref{bounded}, $\overline{\mu}$ is indeed a probability.
Define $\mu\in\mathcal M(f)$ by
 $$\mu=\left(\sum_{\omega\in\mathcal A_{n,\varepsilon}}
 \tau(\omega)\overline{\mu}(\omega)\right)^{-1}\sum_{\omega\in\mathcal A_{n,\varepsilon}}\sum_{i=0}^{\tau(\omega)-1}(f^i)_*(\overline{\mu}|_{\omega}).$$

\begin{sublemma}\label{minus}
$h(\mu)- d^u_\varepsilon\lambda^u(\mu)\leq0$.
\end{sublemma}
\begin{proof}
From the definition of $d^u_\varepsilon$ in \eqref{Feps} it suffices to show 
$ |\lambda^u(\mu)-\beta|<\varepsilon.$
Let $\omega\in\mathcal A_{n,\varepsilon}$ and
$x\in \omega$. Choose
$y\in \omega\cap\gamma^u(\zeta_0)$ such that 
$$\left|\frac{1}{\tau(\omega)}\sum_{i=0}^{\tau(\omega)-1}\log J^u(f^iy)-\beta\right|<\frac{\varepsilon}{2}.$$
Then we have
\begin{align*}\left|\frac{1}{\tau(\omega)}\sum_{i=0}^{\tau(\omega)-1}\log J^u(f^ix)-\beta\right|\leq&
\left| \frac{1}{\tau(\omega)}\sum_{i=0}^{\tau(\omega)-1}\log J^u(f^ix)-\log J^u(f^iy)\right|+
 \left|\frac{1}{\tau(\omega)}\sum_{i=0}^{\tau(\omega)-1}\log J^u(f^iy)-\beta\right|\\
 \leq& \frac{\log D_m}{\tau(\omega)}+\frac{\varepsilon}{2}\leq \frac{\log D_m}{2n}+\frac{\varepsilon}{2}   <\varepsilon.\end{align*}
The upper bound of the first summand follows from Lemma \ref{lyapbdd}.
The third inequality follows from $\tau(\omega)\geq 2n$ in Lemma \ref{globals}(a).
The last one holds for sufficiently large $n$.
Since $\omega\in\mathcal A_{n,\varepsilon}$ and $x\in\omega$ are arbitrary, this implies
 $ |\lambda^u(\mu)-\beta|<\varepsilon$.
 \end{proof}
 

Observe that 
\begin{equation}\label{observe2}\nu_k(\{\underline{a}\})
=\left(\sum_{\underline{b}\in E_k}\exp\left(     S_{k}\Phi(\underline{b})            \right)\right)^{-1}
\exp\left(S_{k}\Phi(\underline{a})\right)
\quad \forall \underline{a}\in E_k.\end{equation}
Hence
\begin{align*}\sum_{\underline{a}\in E_k}\nu_k(\{\underline{a}\})S_{k}\Phi(\underline{a})
&=\sum_{\underline{a}\in E_k}\nu_k(\{\underline{a}\}) \sum_{i=0}^{k-1}\delta_{\sigma^i\underline{a}}(\Phi)\\
&=\left(\sum_{\underline{b}\in E_k}\exp\left(     S_{k}\Phi(\underline{b})            \right)\right)^{-1}
\sum_{\underline{a}\in E_k}\exp\left(S_{k}\Phi(\underline{a})\right) \sum_{i=0}^{k-1}\delta_{\sigma^i\underline{a}}(\Phi)
\\
&=k\int \Phi d\nu_k,\end{align*}
and
\begin{align*}
-\sum_{\underline{a}\in E_k}\nu_k(\{\underline{a}\})
\log\nu_k(\{\underline{a}\})+k\nu_k(\Phi)&=\sum_{\underline{a}\in E_k}\nu_k(\{\underline{a}\})\left(-\log\nu_k(\{\underline{a}\})+
S_{k}\Phi(\underline{a})\right)\\
&=\log\sum_{\underline{a}\in E_k}\exp(S_{k}\Phi(\underline{a})),
\end{align*}
where the last equality follows from taking logs of \eqref{observe2},
rearranging and summing the result for all $\underline{a}\in E_k$.
A slight modification of the argument in \cite[pp.220]{Wal82} shows that for
any integer $p$ with $1\leq p<k$,
\begin{equation}\label{x}
\frac{1}{k}\log\sum_{\underline{a}\in E_k}\exp(S_{k}\Phi(\underline{a}))\leq-\frac{1}{p}\sum_{\underline{a}\in E_p}\nu_k(\{\underline{a}\})
\log\nu_k(\{\underline{a}\})+\nu_k(\Phi)
+\frac{2p\log\#E_p}{k}.\end{equation} 
\begin{sublemma}\label{kconv}
$\int\Phi d\nu_k\to\int\Phi d\nu_0$ as $k\to\infty$.
\end{sublemma}
\begin{proof}
Set $B^c= \Sigma_\ell\setminus B$.
For any $\varepsilon>0$
choose a compact set $K\subset B^c$ such that 
$\nu_0(B^c\setminus K)<\varepsilon.$
Since the set
$\Sigma_\ell\setminus K$ is open and closed, and $\nu_0(B\setminus K)=0$ by Sublemma \ref{bounded},
$\displaystyle{\lim_{k\to\infty}\nu_k(\Sigma_\ell\setminus K)}=\nu_0(\Sigma_\ell\setminus K)=\nu_0(B\setminus K)+\nu_0(B^c\setminus K)=\nu_0(B^c\setminus K)<\varepsilon.$
Hence, for sufficiently large $k$,
$$|\int\Phi d\nu_k-\int \Phi d\nu_0|\leq\left|\int_{K}\Phi d\nu_k-\int_{K}\Phi d\nu_0\right|+
\left|\int_{\Sigma_\ell\setminus K}\Phi d\nu_k-\int_{\Sigma_\ell\setminus K}\Phi d\nu_0\right|<\varepsilon\left(1+\sup_{\underline{a}\in\Sigma_\ell}|\Phi(\underline{a})|\right).\qedhere$$
\end{proof}

Letting $k\to\infty$ and then using Sublemma \ref{kconv},
\begin{equation*}
\limsup_{k\to\infty}\frac{1}{k}\log\sum_{\underline{a}\in E_k}\exp(S_{k}\Phi(\underline{a}))\leq
-\frac{1}{p}\sum_{\underline{a}\in E_p}\nu_0(\{\underline{a}\})
\log\nu_0(\{\underline{a}\})+\int \Phi d\nu_0.
\end{equation*}
Letting $p\to\infty$ we get
\begin{equation}\label{lem2}
\limsup_{k\to\infty}\frac{1}{k}\log\sum_{\underline{a}\in E_k}\exp(S_{k}\Phi(\underline{a}))\leq
h(\sigma;\nu_0)+\int\Phi d\nu_0,
\end{equation}
where $h(\sigma;\nu_0)$ denotes the entropy of $\nu_0\in\mathcal M(\sigma)$.
We estimate the left-hand-side of \eqref{lem2} from below.

\begin{sublemma}\label{mary1}
Let $\underline{a}=\{a_i\}_{i\in\mathbb Z}\in E_k$ be such that $a_0=1$. Then:

\begin{itemize}

\item[(a)] $\frac{\exp(S_{k-1}\Phi(\underline{a}))}{   \exp(S_{k-1}\Phi(\underline{b})) }\geq D_m^{-d^u_\varepsilon}$
for every $\underline{b}=\{b_i\}_{i\in\mathbb Z}\in\Sigma_\ell$ such that
$a_i=b_i$ for every $0\leq i< k-1$.

\item[(b)] $\frac{\exp(S_{0}\Phi(\underline{a}))}{{\rm length}(\omega^u(a_{k-1}))^{ d^u_\varepsilon}}\geq D_m^{-2d^u_\varepsilon}$.
\end{itemize}

\end{sublemma}
\begin{proof}
It suffices to show $\pi\underline{a}\subset G_m$.
Indeed, if this holds, then
since $\pi\underline{a}$ and $\pi\underline{b}$ are contained in the same proper rectangle
 with inducing time $T_{k-1}>m$, 
 Lemma \ref{lyapbdd} gives (a).
 (b) also follows from  Lemma \ref{lyapbdd}.

Set $T_j=\sum_{i=0}^{j-1}\tau(a_i)$ for $1\leq j\leq k+1$.
Since $\pi\underline{a}$ is a periodic point of period $T_k$, it suffices to show 
\begin{equation}\label{ama}d_{\rm crit}(f^n(\pi\underline{a}))> b^\frac{n}{10}\  \text{ for every  }\ m\leq n\leq   T_k+m-1.\end{equation}
The inequality in \eqref{ama} for $m\leq n\leq T_1-1$
is a consequence of Lemma \ref{nonint}.
For $T_j\leq n\leq T_{j+1}-1$ $(j=1,\ldots,k)$,
Lemma \ref{globals}(c) and $n\geq \tau(a_0)\geq \tau(a_{j-1})$ from \eqref{align} yield 
$$d_{\rm crit}(f^n(\pi\underline{a}))\geq e^{-10\tau(a_{j-1})}
\geq e^{-10\tau(a_0)}\geq e^{-10n}> b^{\frac{n}{10}}.$$
This covers all $n$.
\end{proof}

Set $E_k'=\{\underline{a}\in E_k\colon a_0=1\}.$
Let 
$\underline{a}\in E_k'$, $\underline{b}\in E_{k-1}'$ be such that
$a_i=b_i$ for every $0\leq i< k-1$.
 By Sublemma \ref{mary1}, 
 $$\frac{\exp(S_{k}\Phi(\underline{a}))}{   \exp(S_{k-1}\Phi(\underline{b})) }=\frac{\exp(S_{k-1}\Phi(\underline{a}))}{   \exp(S_{k-1}\Phi(\underline{b})) }\exp(S_{0}\Phi(\sigma^{k-1}\underline{a}))\geq D_m^{-3d^u_\varepsilon}{\rm length}(\omega^u)^{d^u_\varepsilon}.$$
Using this inequality repeatedly gives
\begin{align*}
\sum_{\underline{a}\in E_k}\exp(S_k\Phi(\underline{a}))&>
\sum_{\underline{a}\in E_k'} \exp(S_{k}\Phi(\underline{a}))=
\sum_{\underline{b}\in E_{k-1}'} \exp(S_{k-1}\Phi(\underline{b}))\sum_{\stackrel{\underline{a}\in E_k'}
{a_i=b_i\ 0\leq \forall i< k-1}}\frac{ \exp(S_{k}\Phi(\underline{a}))}
{ \exp(S_{k-1}\Phi(\underline{b}))}\\
&\geq\sum_{\underline{b}\in E_{k-1}'} \exp(S_{k-1}\Phi(\underline{b}))\cdot D_m^{-3d^u_\varepsilon}\sum_{\omega\in\mathcal A_{n,\varepsilon}}
{\rm length}(\omega^u)^{ d^u_\varepsilon}\\
&\geq\cdots\geq
\sum_{\underline{b}\in E_{1}'} \exp(S_{0}\Phi(\underline{b}))\left(D_m^{-3d^u_\varepsilon}\sum_{\omega\in\mathcal A_{n,\varepsilon}}
{\rm length}(\omega^u)^{d^u_\varepsilon}\right)^{k-1}\\
&\geq  \left(D_m^{-3d^u_\varepsilon}\sum_{\omega\in\mathcal A_{n,\varepsilon}}
{\rm length}(\omega^u)^{d^u_\varepsilon}\right)^{k}.\end{align*}
Hence
\begin{equation}\label{lem1}\liminf_{k\to\infty}\frac{1}{k}
\log\sum_{\underline{a}\in E_k}\exp(S_k\Phi(\underline{a}))\geq\log\sum_{\omega^u\in\mathcal A^u_{n,\varepsilon}}
{\rm length}(\omega^u)^{d^u_\varepsilon}- 3d^u_\varepsilon\log D_m.\end{equation}
Putting \eqref{lem2}  \eqref{lem1} together and then using Lemma \ref{minus} yield
\begin{align*}
\frac{1}{n}\log\sum_{\omega\in\mathcal A^u_{n,\varepsilon}}
{\rm length}(\omega^u)^{d^u_\varepsilon}&\leq
\frac{1}{n}(h(\sigma;\nu_0)+\nu_0(\Phi))+\frac{3}{n} d^u_\varepsilon\log D_m\\
&=\frac{1}{n}(h(\mu)- d^u_\varepsilon\lambda^u(\mu))\sum_{\omega\in\mathcal A_{n,\varepsilon}}\tau(\omega)\overline{\mu}(\omega)+ \frac{3}{n} d^u_\varepsilon\log D_m\\
&   \leq \frac{3}{n} d^u_\varepsilon\log D_m.\end{align*}
This implies \eqref{pperd}, and hence finishes the proof of Lemma \ref{c}.
\qed

\subsection{Properties of the Lyapunov spectrum}\label{continuous}
We now prove Theorem C.
\medskip

\noindent{\it Proof of Theorem C(a).}
The upper semi-continuity follows from the formula in Theorem B.
We derive a contradiction assuming $L^u$ is not lower semi-continuous at  
 a point $\beta\in I$.
Then there exist $\varepsilon>0$ and a monotone sequence $\{\beta_n\}_n$
  converging to $\beta$ such that
$L^u(\beta_n)\leq L^u(\beta)-\varepsilon.$

 If $\beta=\lambda^u_M$, then $\mu\in\mathcal M(f)$ with $\lambda^u(\mu)<\beta$. Choose a sequence $\{\mu_n\}_n$ in $\mathcal M(f)$ such that
$h(\mu_n)/\lambda^u(\mu_n)\geq F(\beta)-\varepsilon/4$ and
$\lambda^u(\mu_n)\to\beta$ as $n\to\infty$.
Taking a subsequence if necessary we may assume
$\beta_n\leq \lambda^u(\mu_{n})$. 
For those sufficiently large $n$ such that $\lambda^u(\mu)\leq\beta_n$,
choose $t_n\in[0,1]$ with
$(1-t_n)\lambda^u(\mu)+t_n\lambda^u(\mu_{n})=\beta_n$.
Then
\begin{align*} L^u(\beta)-\varepsilon\geq L^u(\beta_n)&=L^u((1-t_n)\lambda^u(\mu)+t_n\lambda^u(\mu_{n}))\\
&\geq \frac{h((1-t_n)\mu+t_n\mu_{n})}{\lambda^u((1-t_n)\mu+t_n\mu_{n})}
=\frac{(1-t_n)h(\mu)+t_nh(\mu_{n})}{(1-t_n)\lambda^u(\mu)+t_n\lambda^u(\mu_{n})}\geq
L^u(\beta)- \varepsilon /2.\end{align*}
The second inequality follows from $t_n\to1$ and $\lambda^u(\mu_{n})\geq\lambda_m^u>0$.
This yields a contradiction.

If $\beta=\lambda^u_m$, then we replace $\mu$ by $\mu'$ with $\lambda^u(\mu')>\beta$ and proceed in the same way.
The remaining case $\beta\in(\lambda^u_m,\lambda^u_M)$ is covered by the same 
argument. \qed
\medskip

\noindent{\it Proof of Theorem C(b).}
Follows from the next
 \begin{lemma}\label{monotone}
For all $\beta,\beta'\in I$ with $\beta<\beta'$ and $0\leq t\leq 1$,
$$\min\left\{L^u(\beta), L^u(\beta')\right\}\leq L^u(t\beta+(1-t)\beta').$$
\end{lemma}
 \begin{proof}
FromTheorem B, for any $\varepsilon>0$ there exist $\mu$, $\mu' \in\mathcal M(f)$
such that
$L^u(\beta)-\varepsilon< h(\mu)/\lambda^u(\mu)$,
$L^u(\beta')-\varepsilon< h(\mu')/\lambda^u(\mu')$
and $|\lambda^u(\mu)-\beta|< \varepsilon$,
$|\lambda^u(\mu')-\beta'|< \varepsilon$.
Then
$$\min\left\{L^u(\beta), L^u(\beta')\right\}
<\varepsilon+\min\left\{\frac{h(\mu)}{\lambda^u(\mu)},
\frac{h(\mu')}{\lambda^u(\mu')}\right\}.$$
Set $\nu=t\mu+(1-t)\mu'$. 
It is easy to see that
the minimum of the right-hand-side is
$\leq h(\nu)/\lambda^u(\nu)$. 
Letting $\varepsilon\to0$ yields the desired inequality.
\end{proof}
\medskip

\noindent{\it Proof of Theorem C(c).} 
Contained in Lemma \ref{complete}. \qed
\medskip

\noindent{\it Proof of Theorem C(d).} 
The ``if" part follows from Theorem B.
To show the ``only if" part,
let $\beta\in I$ be such that $L^u(\beta)=t^u$.
Theorem B allows us to
choose a sequence $\{\mu_n\}_n$ in $\mathcal M(f)$ such that $h(\mu_n)/\lambda^u(\mu_n)\to t^u$ and
$\lambda^u(\mu_n)\to\beta$ as $n\to\infty$. Choosing a subsequence if necessary we may assume $\mu_n\to\mu\in\mathcal M(f)$. 
Write $\mu= u\delta_Q+(1-u)\nu$, $0\leq u\leq1$, $\nu\{Q\}=0$.
Since $h(\delta_Q)=0$, 
the upper semi-continuity of entropy \cite[Corollary 3.2]{SenTak1} 
implies $u\neq1$ and
$\displaystyle{\limsup_{n\to\infty}}h(\mu_n)\leq h(\mu)=(1-u)h(\nu).$
On the other hand, \cite[Lemma 4.3]{SenTak1} gives
$\displaystyle{\liminf_{n\to\infty}}\lambda^u(\mu_n)\geq (1-u)\lambda^u(\nu).$
If $u\neq0$, then this inequality would be strict, and so
$$\frac{h(\nu)}{\lambda^u(\nu)}>\frac{\displaystyle{\limsup_{n\to\infty}h(\mu_n)}}
{\displaystyle{\liminf_{n\to\infty}\lambda^u(\mu_n)}}
\geq\lim_{n\to\infty}\frac{h(\mu_n)}{\lambda^u(\mu_n)}=t^u,$$ 
which yields
$P(t^u)>0$, a contradiction.
Hence $u=0$.
 \cite[Lemma 4.4]{SenTak1} gives $\lambda^u(\mu_n)\to\lambda^u(\mu)$,
and so
$h(\mu_n)\to t^u\lambda^u(\mu)$ and $t^u\lambda^u(\mu)\leq h(\mu)$.
From the uniqueness of the equilibrium measure
for the potential $-t^u\log J^u$ \cite[Theorem A]{SenTak2},
$\mu=\mu_{t^u}$
and 
$\beta=\lambda^u(\mu_{t^u})$.
\qed

\subsection{Hausdorff dimension of the set of irregular points}\label{irregular}
We now prove Theorem D.
\medskip

\noindent{\it Proof of Theorem D.}
For any $\varepsilon>0$ choose $\mu,\nu\in\mathcal M^e(f)$ with
$\lambda^u(\mu)>\lambda^u(\nu)$ and
 $h(\mu)/\lambda^u(\mu)$,  $h(\nu)/\lambda^u(\nu)\geq t^u-\varepsilon$.
 Choose sequences $\{\mu_{n}\}_{n=1}^\infty$, $\{\nu_{n}\}_{n=1}^\infty$ in $\mathcal M^e(f)$
with $\lambda^u(\mu_n)\to\lambda^u(\mu)$ and $\lambda^u(\nu_n)\to\lambda^u(\nu)$ as $n\to\infty$.
Define $\xi_{n}\in\mathcal M^e(f)$ by
$$\xi_n=\begin{cases}\mu_n\text{ for $n$ odd;}\\\nu_n\text{ for $n$ even.}\end{cases}$$
A slight modification of the proof of Lemma \ref{lowd}
applied to the sequence $\{\xi_{n}\}_{n=1}^\infty$ 
yields
a set $\Gamma\subset \Omega^u$ such that
$\bar{\lambda}^u(x)=\lambda^u(\mu)$ and $\underline{\lambda}^u(x)=\lambda^u(\nu)$ for all
$x\in\Gamma,$
and
 $$\dim_H^u(\Gamma)\geq\min\left\{\frac{h(\mu)}{\lambda^u(\mu)},   \frac{h(\mu)}{\lambda^u(\mu)}   \right\}.$$
 Hence $\Gamma\subset\hat\Omega^u$ and $\dim_H^u(\hat\Omega^u)\geq
 \dim_H^u(\Gamma)\geq t^u-\varepsilon$.
Letting $\varepsilon\to0$ we obtain Theorem D. \qed

\subsection*{Acknowledgments}
Partially supported by the Grant-in-Aid for Young Scientists (B) of the JSPS, Grant No.23740121.

\bibliographystyle{amsplain}

\end{document}